\documentclass[a4paper]{amsart}
\usepackage{amsmath,amssymb}
\usepackage{amsthm}
\usepackage{latexsym}
\usepackage[latin1]{inputenc}
\usepackage{graphicx}
\usepackage{array}
\usepackage{url}

\usepackage[all]{xy}

\usepackage[colorlinks=true, linkcolor={black}, citecolor={black}]{hyperref}

\DeclareMathOperator{\Gr}{Gr}
\DeclareMathOperator{\LL}{\mathrm{L}\!}
\DeclareMathOperator{\RR}{\mathrm{R}}
\DeclareMathOperator{\Pic}{\mathrm{Pic}}
\DeclareMathOperator{\coker}{\mathrm{coker}}
\DeclareMathOperator{\Proj}{\mathrm{Proj}}
\DeclareMathOperator{\VB}{\mathrm{VB}}
\DeclareMathOperator{\coh}{\mathrm{coh}}
\DeclareMathOperator{\Qcoh}{\mathrm{Qcoh}}
\DeclareMathOperator{\Der}{\mathrm{D}}
\DeclareMathOperator{\codim}{\mathrm{codim}}
\DeclareMathOperator{\id}{\mathrm{id}}

\newcommand{\cF}{{\mathcal F}}
\newcommand{\cJ}{{\mathcal J}}
\newcommand{\cO}{{\mathcal O}}
\newcommand{\cS}{{\mathcal S}}
\newcommand{\bbA}{{\mathbb A}}
\newcommand{\bbZ}{{\mathbb Z}}
\newcommand{\bbP}{{\mathbb P}}
\newcommand{\Sch}{\mathcal Sch}
\newcommand{\kdim}{\dim}
\newcommand{\Kdim}[1]{\kdim(#1)}
\newcommand{\Wcoh}{\tilde{\mathrm{W}}}
\newcommand{\W}{\mathrm{W}}
\newcommand{\bord}{\partial}
\newcommand{\can}{\omega}
\newcommand{\Line}{\cO}
\newcommand{\Bl}{Bl}
\newcommand{\Exc}{E}
\newcommand{\RHom}{\mathrm{RHom}}
\newcommand{\Db}{\Der^{\mathrm{b}}}
\newcommand{\Dperf}{\Der_\mathrm{perf}}
\newcommand{\Dbc}{\Db_{\mathrm{coh}}}
\newcommand{\isoto}{\buildrel \sim\over\to}
\newcommand{\too}{\mathop{\longrightarrow}}
\newcommand{\tooo}[1]{\mathop{\vcenter{\hbox to #1em{\hrulefill}}\kern-5pt\to}}
\newcommand{\inv}{^{-1}}
\newcommand{\restr}[1]{_{|_{\scriptstyle #1}}}
\newcommand{\equalby}[2]{\stackrel{#2}{#1}}
\newcommand{\eg}{{e.g.}}
\newcommand{\ie}{i.e.\ }
\newcommand{\loccit}{{loc.\ cit.}}
\newcommand{\smallmatrice}[1]{\bigl(\begin{smallmatrix} #1 \end {smallmatrix}\bigr)}

\swapnumbers 
\theoremstyle{plain}
\newtheorem{thm}{Theorem}[section]
\newtheorem*{thm*}{Theorem}
\newtheorem{mainthmsing}[thm]{Main Theorem in the non-regular case}
\newtheorem{mainthmreg}[thm]{Main Theorem}
\newtheorem{mainlem}[thm]{Main Lemma}
\newtheorem{mainlemcodimone}[thm]{Main Lemma in codimension one}
\newtheorem{lem}[thm]{Lemma}

\newtheorem{prop}[thm]{Proposition}

\theoremstyle{definition}
\newtheorem{defi}[thm]{Definition}
\newtheorem{nota}[thm]{Notation}
\newtheorem{setup}[thm]{Setup}
\newtheorem{hypo}[thm]{Hypothesis}
\newtheorem{rem}[thm]{Remark}
\newtheorem{exa}[thm]{Example}

\title[Connecting homomorphism for Witt groups]{{\bf Geometric description of the connecting homomorphism for Witt groups}}
\author[P. Balmer and B. Calm{\`e}s]{Paul Balmer and Baptiste Calm{\`e}s}

\address{Paul Balmer, Department of Mathematics, UCLA, Los Angeles, CA 90095-1555, USA}
\email{balmer@math.ucla.edu}
\urladdr{http://www.math.ucla.edu/$\sim$balmer}

\address{Baptiste Calm\`es, Laborartoire de Mathématiques de Lens, Faculté des Sciences Jean Perrin, Université d'Artois, 62307 Lens, France}
\email{calmes@math.jussieu.fr}
\urladdr{http://www.people.math.jussieu.fr/~calmes}

\subjclass{19G12, 11E81}
\thanks{Research supported by SNSF grant~PP002-112579 and NSF grant 0654397.}

\date{\today}

\begin{document}

\bibliographystyle{amsplain}

\begin{abstract}
We give a geometric setup in which the connecting homomorphism in
the localization long exact sequence for Witt groups decomposes as
the pull-back to the exceptional fiber of a suitable blow-up
followed by a push-forward.
\end{abstract}

\maketitle



\section{Introduction}

Witt groups form a very interesting cohomology theory in algebraic
geometry. (For a survey, see~\cite{Balmer05a}.) Unlike the better
known $K$-theory and Chow theory, Witt theory is not \emph{oriented}
in the sense of Levine-Morel~\cite{LevineMorel07} or
Panin~\cite{Panin04}, as already visible on the non-standard
projective bundle theorem, see Arason~\cite{Arason80} and
Walter~\cite{Walter03}. Another way of expressing this is that
push-forwards do not exist in sufficient generality for Witt groups.
This ``non-orientability'' can make computations unexpectedly
tricky. Indeed, the Witt groups of such elementary schemes as
Grassmann varieties will appear for the first time in the companion
article~\cite{BalmerCalmes08pp2}, whereas the corresponding
computations for oriented cohomologies have been achieved more than
35 years ago in~\cite{Laksov72}, using the well-known cellular
decomposition of Grassmann varieties. See also~\cite{Nenashev06} for
general statements on cellular varieties.

In oriented theories, there is a very useful computational technique,
recalled in Theorem~\ref{standard_thm} below, which allows inductive
computations for families of cellular varieties. Our paper
originates in an attempt to extend this result to the non-oriented
setting of Witt theory. Roughly speaking, such an extension is
possible ``half of the time''. In the remaining ``half'', some specific ideas must come in and reflect the truly
non-oriented behavior of Witt groups. To explain this rough statement, let us fix
the setup, which will remain valid for the entire paper.

\begin{setup}\label{setup}%
We denote by $\Sch$ the category of separated connected noetherian $\bbZ[\frac{1}{2}]$-scheme.
Let $X,Z\in \Sch$ be schemes and let $\iota:Z\hookrightarrow X$ be a
regular closed immersion of codimension~$c\geq 2$. Let $\Bl=\Bl_ZX$
be the blow-up of $X$ along~$Z$ and $\Exc$ the exceptional fiber.
Let $U=X- Z\cong \Bl-\Exc$ be the unaltered open complement. We have
a commutative diagram
\begin{equation}
\label{blowup_eq}%
\vcenter{\xymatrix{ \ Z\  \ar@{^(->}[r]^-{\displaystyle \iota}
& X
& \ U\ \ar@{_(->}[l]_-{\displaystyle \upsilon}
 \ar@{^(->}[ld]^-{\displaystyle \tilde\upsilon}
\\
\ \Exc\ \ar@{^(->}[r]_-{\displaystyle \tilde \iota}
\ar[u]^-{\displaystyle \tilde \pi}
& \Bl \ar[u]^-{\displaystyle \pi}
}}
\end{equation}
with the usual morphisms.
\end{setup}

Consider now a cohomology theory with supports, say~$H^*$
\begin{equation}
\label{LES-H*_eq}%
\cdots \too^{\bord} H^*_Z(X) \too H^*(X) \too^{\upsilon^*} H^*(U) \too^{\bord}
H^{*+1}_Z(X) \too \cdots
\end{equation}
In this paper we shall focus on the case of Witt groups $H^*=\W^*$
but we take inspiration from $H^*$ being an oriented cohomology
theory. Ideally, we would like conditions for the vanishing of the
connecting homomorphism $\bord=0$ in the above localization long
exact sequence. Even better would be conditions for the restriction
$\upsilon^*$ to be split surjective. When $H^*$ is an oriented
theory, there is a well-known hypothesis under which such a
splitting actually exists, namely\,:

\begin{hypo}
\label{basic_hypo}%
Assume that there exists an auxiliary morphism $\tilde\alpha: \Bl\to Y$
\begin{equation}
\label{hypo_eq}%
\vcenter{\xymatrix{
\ Z\  \ar@{^(->}[r]^-{\iota}
& X
& \ U\ \ar@{_(->}[l]_-{\upsilon}
 \ar@{_(->}[ld]_-{\tilde\upsilon}
 \ar@{..>}[d]^-{\displaystyle\alpha}
\\
\ \Exc\ \ar@{^(->}[r]_-{\tilde \iota} \ar[u]^-{\tilde \pi}
& \Bl \ar[u]^-{\pi} \ar@{..>}[r]_-{\displaystyle\tilde\alpha}
& Y}}
\end{equation}
such that $\alpha:=\tilde\alpha\circ\tilde\upsilon:U\to Y$ is an \emph{$\bbA^{\!*}$-bundle}, \ie every point of~$Y$ has a Zariski neighborhood over which $\alpha$ is isomorphic to a trivial $\bbA^{\!r}$-bundle, for some $r\geq0$. See Ex.\,\ref{Grass_exa} for an explicit example with $X$, $Y$ and $Z$ being Grassmann varieties.
\end{hypo}

\begin{thm}[The oriented technique]
\label{standard_thm}%
Under Setup~\ref{setup} and Hypothesis~\ref{basic_hypo}, assume $X$,
$Y$ and $Z$ regular. Assume the cohomology theory $H^*$ is
homotopy invariant for regular schemes and \emph{oriented}, in that
it admits push-forwards along proper morphisms satisfying flat
base-change. Then, the restriction $\upsilon^*:H^*(X)\to H^*(U)$ is
split surjective with explicit section
$\pi_*\circ\tilde\alpha^*\circ(\alpha^*)\inv$, where
$\pi_*:H^*(\Bl)\to H^*(X)$ is the push-forward. Hence the connecting
homomorphism $\bord:H^*(U)\to H^{*+1}_Z(X)$ vanishes and the above
localization long exact sequence~\eqref{LES-H*_eq} reduces to split
short exact sequences $0\to H^*_Z(X) \to H^*(X) \to H^*(U) \to0$.
\end{thm}

\begin{proof}
By homotopy invariance, we have $\alpha^*:H^*(Y)\isoto
H^*(U)$. By base-change, $\upsilon^*\circ\pi_*=\tilde\upsilon^*$ and
since $\tilde\upsilon^*\circ\tilde\alpha^*=\alpha^*$, we have
$\upsilon^*\circ\pi_*\circ\tilde\alpha^*\circ(\alpha^*)\inv=\id$.
\end{proof}

The dichotomy between the cases where the above technique extends to Witt
groups and the cases where is does not, comes from the duality. To
understand this, recall that one can consider Witt groups $\W^*(X,L)$
with duality twisted by a line bundle $L$ on the scheme~$X$.
Actually only the class of the twist $L$ in $\Pic(X)/2$ really
matters since we have square-periodicity isomorphisms for all $M\in\Pic(X)$
\begin{equation}
\label{square_eq}%
\W^*(X,L)\cong \W^*(X,L\otimes M^{\otimes2})\,.
\end{equation}

Here is a condensed form of our Theorem~\ref{extend_thm} and Main Theorem~\ref{mainreg_thm} below\,:

\begin{thm}
\label{super_thm}%
Under Hypothesis~\ref{basic_hypo}, assume $X$, $Y$ and $Z$ regular. Let $L\in\Pic(X)$. Then there exists an integer $\lambda(L)\in\bbZ$ (defined by~\eqref{lambda_eq} below) such that\,:
\begin{enumerate}
\item[(A)] If $\lambda(L)\equiv c-1 \!\mod 2$ then the restriction $\upsilon^*:\W^*(X,L)\to \W^*(U,L\restr{U})$ is split surjective with a section given by the composition $\pi_*\circ\tilde\alpha^*\circ(\alpha^*)\inv$. Hence the connecting homomorphism $\W^*(U,L\restr{U})\too\limits^{\bord} \W^{*+1}_Z(X,L)$ vanishes and the localization long exact sequence reduces to split short exact sequences
$$
0\too \W^*_Z(X,L)\too \W^*(X,L)\too \W^*(U,L\restr{U})\too 0\,.
$$
\item[(B)]
If $\lambda(L)\equiv c \!\mod 2$ then the connecting homomorphism $\bord$ is equal to a composition of pull-backs and push-forwards\,: $\bord=\iota_*\circ\tilde\pi_*\circ\tilde\iota^*\circ\tilde\alpha^*\circ(\alpha^*)\inv$.
\end{enumerate}
\end{thm}

This statement requires some explanations. First of all, note that
we have used push-forwards for Witt groups, along $\pi:\Bl\to X$
in~(A) and along $\tilde\pi:\Exc\to Z$ and $\iota:Z\to X$ in~(B). To
explain this, recall that the push-forward in Witt theory is only
conditionally defined. Indeed, given a proper morphism $f: X'\to X$
between (connected) regular schemes and given a line bundle
$L\in\Pic(X)$, the push-forward homomorphism does not map
$\W^*(X',f^*L)$ into $\W^*(X,L)$, as one could naively expect, but the
second author and Hornbostel~\cite{Calmes08b_pre} showed that
Grothendieck-Verdier duality yields a twist by the relative
canonical line bundle~$\can_f\in\Pic(X')$\,:
\begin{equation}
\label{push-forward_eq}%
\W^{i+\Kdim{f}}\big(\,X'\,,\,\can_f\otimes f^*L \,\big)\,\too^{f_*}\, \W^i(\,X\,,\,L\,)\,.
\end{equation}
Also note the shift by the relative dimension, $\Kdim{f}:=\kdim X'-\kdim
X$, which is not problematic, since we can always replace $i\in\bbZ$
by $i-\Kdim{f}$.

More trickily, if you are given a line bundle $M\in \Pic(X')$ and if you need a push-forward $\W^*(X',M)\to \W^{*-\Kdim{f}}(X,?)$ along~$f:X'\to X$, you first need to check that $M$ is of the form $\can_f\otimes f^*L $ for some $L\in \Pic(X)$, at least module squares. \emph{Otherwise, you simply do not know how to push-forward}. This is precisely the source of the dichotomy of Theorem~\ref{super_thm}, as explained in Proposition~\ref{dichotomy_prop} below.

\smallbreak

At the end of the day, it is only possible to transpose to Witt groups the oriented technique of Theorem~\ref{standard_thm} when the push-forward $\pi_*$ exists for Witt groups. But actually, the remarkable part of Theorem~\ref{super_thm} is Case~(B), that is our Main Theorem~\ref{mainreg_thm} below, which gives a description of the connecting homomorphism~$\bord$ when we cannot prove it zero by the oriented method. This is the part where the non-oriented behavior really appears. See more in Remark~\ref{punch_rem}. Main Theorem~\ref{mainreg_thm} is especially striking since the original definition of the connecting homomorphism given in~\cite[\S\,4]{Balmer00} does not have such a geometric flavor of pull-backs and push-forwards but rather involves abstract techniques of triangulated categories, like symmetric cones, and the like. Our new geometric description is also remarkably simple to use in applications, see~\cite{BalmerCalmes08pp2}. Here is the example in question.

\begin{exa}
\label{Grass_exa}%
Let $k$ be a field of characteristic not~2. (We describe flag
varieties over $k$ by giving their $k$-points, as is customary.) Let
$1\leq d\leq n$. Fix a codimension one subspace $k^{n-1}$ of~$k^n$.
Let $X=\Gr_d(n)$ be the Grassmann variety of $d$-dimensional subspaces $V_d\subset k^{n}$ and let $Z\subset X$ be the closed subvariety of those subspaces $V_d$ contained in~$k^{n-1}$. The open complement $U=X-Z$ consists of those $V_d\not\subset k^{n-1}$. For such $V_d\in U$, the subspace $V_d\cap k^{n-1}\subset k^{n-1}$ has dimension~$d-1$. This construction defines an $\bbA^{n-d}$-bundle $\alpha:U\to Y:=\Gr_{d-1}(n-1)$, mapping $V_d$ to $V_d\cap k^{n-1}$. This situation relates the Grassmann variety $X=\Gr_d(n)$ to the smaller ones $Z=\Gr_d(n-1)$ and $Y=\Gr_{d-1}(n-1)$. Diagram~\eqref{blowup_eq} here becomes
$$
\vcenter{\xymatrix@C=2em{
\kern-1em \Gr_d(n-1)\  \ar@{^(->}[r]^-{\iota}
& \Gr_d(n)
& \ U\ \ar@{_(->}[l]_-{\upsilon}
 \ar@{^(->}[ld]_-{\tilde\upsilon}
 \ar[d]^-{\alpha}
\\
\ \Exc\ \ar@{^(->}[r]_-{\tilde \iota} \ar[u]^-{\tilde \pi}
& \Bl \ar[u]^-{\pi} \ar[r]_-{\tilde\alpha}
& \Gr_{d-1}(n-1)\,.\kern-1em}}
$$
The blow-up~$\Bl$ is the variety of pairs of subspaces
$V_{d-1}\subset V_d$ in $k^{n}$, such that $V_{d-1}\subset k^{n-1}$.
The morphisms $\pi:\Bl\to X$ and $\tilde \alpha:\Bl\to Y$ forget
$V_{d-1}$ and $V_d$ respectively. The morphism $\tilde\upsilon$ maps
$V_d\not\subset k^{n-1}$ to the pair $(V_d\cap k^{n-1})\subset V_d$.

Applying Theorem~\ref{standard_thm} to this situation, Laksov~\cite{Laksov72}
computes the Chow groups of Grassmann varieties by induction. For
Witt groups though, there are cases where the restriction
$\W^*(X,L)\to \W^*(U,L\restr{U})$ is not surjective (see~\cite[Cor.\,6.7]{BalmerCalmes08pp2}). Nevertheless,
thank to our geometric description of the connecting homomorphism,
we have obtained a complete description of the Witt groups of
Grassmann varieties, for all shifts and all twists, to appear
in~\cite{BalmerCalmes08pp2}. In addition to the present techniques,
our computations involve other ideas, specific to Grassmann
varieties, like Schubert cells and desingularisations thereof, plus
some combinatorial bookkeeping by means of special Young diagrams.
Including all this here would misleadingly hide the
simplicity and generality of the present paper. We therefore chose
to publish the computation of the Witt groups of Grassmann varieties
separately in~\cite{BalmerCalmes08pp2}.
\end{exa}

\goodbreak

The paper is organized as follows. Section~\ref{reg_sec} is
dedicated to the detailed explanation of the above dichotomy and the
proof of the above Case~(A), see Theorem~\ref{extend_thm}. We also
explain Case~(B) in our Main Theorem~\ref{mainreg_thm} but its proof
is deferred to Section~\ref{main_sec}. The whole
Section~\ref{reg_sec} is written, as above, under the assumption
that all schemes are regular. This assumption simplifies the
statements but can be removed at the price of introducing dualizing
complexes and coherent Witt groups, which provide the natural
framework over non-regular schemes. This generalization is the
purpose of Section~\ref{sing_sec}. There, we even drop the auxiliary
Hypothesis~\ref{basic_hypo}, \ie the dotted part of
Diagram~\eqref{hypo_eq}. Indeed, our Main Lemma~\ref{main_lem} gives
a very general description of the connecting homomorphism applied to
a Witt class over~$U$, if that class comes from the blow-up~$\Bl$
via restriction~$\tilde\upsilon^*$. The proof of Main
Lemma~\ref{main_lem} occupies Section~\ref{codim1_sec}. Finally,
Hypothesis~\ref{basic_hypo} re-enters the game in
Section~\ref{main_sec}, where we prove our Main
Theorem~\ref{mainreg_thm} as a corollary of a non-regular
generalization given in Theorem~\ref{main_thm}. For the convenience
of the reader, we gathered in Appendix~\ref{app} the needed results
about Picard groups, canonical bundles and dualizing complexes,
which are sometimes difficult to find in the literature. The
conscientious reader might want to start with that appendix.

\section{The regular case}
\label{reg_sec}%

We keep notation as in Setup~\ref{setup} and we assume all schemes to be regular. This section can also be considered as an expanded introduction.

As explained after Theorem~\ref{super_thm} above, we have to decide when the push-forward along $\pi:\Bl\to X$ and along $\tilde\pi:\Exc\to Z$ exist. By~\eqref{push-forward_eq}, we need to determine the canonical line bundles $\can_\pi\in\Pic(\Bl)$ and~$\can_{\tilde\pi}\in\Pic(\Exc)$. This is classical and is recalled in Appendix~\ref{app}. First of all, Proposition~\ref{Pic_codim2_prop} gives
$$
\Pic\left(\vcenter{\xymatrix{
\ Z\
 \ar@{^(->}[r]^-{\iota}
& X
& \ U\
 \ar@{_(->}[l]_-{\upsilon}
 \ar@{_(->}[ld]^-{\tilde\upsilon}
\\
\ \Exc\
 \ar@{^(->}[r]_-{\tilde \iota}
 \ar[u]^-{\tilde \pi}
& \Bl \ar[u]^-{\pi}
}}\right)
\ \cong\
\vcenter{
\xymatrix{
\Pic(Z)
 \ar[d]_-{\smallmatrice{1\\0}}
& \Pic(X)
 \ar[l]_-{\iota^*} \ar@{=}[r] \ar[d]_-{\smallmatrice{1\\0}}
& \Pic(X)\,.\!
\\
\Pic(Z)\oplus\bbZ
& \Pic(X)\oplus\bbZ \ar[l]_-{\smallmatrice{\iota^*&0\\0&1}} \ar[ru]_-{( 1 \; 0)}
}}
$$
The $\bbZ$ summands in $\Pic(\Bl)$ and $\Pic(\Exc)$ are generated by $\Line(\Exc)=\cO_{\Bl}(-1)$ and $\Line(\Exc)\restr{E}=\cO_{\Exc}(-1)$ respectively. Then Proposition~\ref{can_prop} gives the wanted
\begin{equation}
\label{can_eq}%
\begin{array}{lcc}
\can_{\pi}=(0,c-1) & \text{in}\quad \Pic(X)\oplus\bbZ\cong\Pic(\Bl) & \text{and}
\\[.5em]
\can_{\tilde{\pi}}=(-\can_{\iota},c) & \text{in}\quad \Pic(Z)\oplus \bbZ\cong\Pic(\Exc)\,. &
\end{array}
\end{equation}

So, statistically, picking a line bundle $M\in\Pic(\Bl)$ at random, there is a 50\% chance of being able to push-forward $\W^*(\Bl,M)\to \W^*(X,L)$ along $\pi$ for some suitable line bundle $L\in\Pic(X)$. To justify this, observe that
$$
\coker\big(\xymatrix@C=1.5em{\Pic(X)\ar[r]^-{\displaystyle\pi^*}&\Pic(\Bl)}\big)\big/2
\ \cong\
\bbZ/2
$$
and tensoring by $\can_\pi$ is a bijection, so half of the elements of $\Pic(Bl)/2$ are of the form $\can_\pi \otimes \pi^*(L)$. The same probability of 50\% applies to the push forward along $\tilde\pi:\Exc\to Z$ but interestingly in complementary cases, as we summarize now.

\begin{prop}
\label{dichotomy_prop}%
With the notation of~\ref{setup}, assume $X$ and $Z$ regular. Recall that $c=\codim_X(Z)$. Let $M\in\Pic(\Bl)$. Let $L\in\Pic(X)$ and $\ell\in\bbZ$ be such that $M=(L,\ell)$ in~$\Pic(\Bl)=\Pic(X)\oplus\bbZ$, that is, $M=\pi^*L\otimes\Line(E)^{\otimes \ell}$.
\begin{enumerate}
\item[(A)]
If $\ell\equiv c-1\!\mod 2$, we can push-forward along~$\pi:\Bl\to X$, as follows\,:
$$
\W^*(\Bl,M)\cong \W^*(\Bl,\can_\pi\otimes\pi^*L)\tooo{2}^{\pi_*} \W^*(X,L)\,.
$$
\item[(B)]
If $\ell\equiv c\!\mod 2$, we can push-forward along~$\tilde\pi:\Exc\to Z$, as follows\,:
$$
\W^*(\Exc,M\restr{\Exc})\cong \W^*\big(\Exc,\can_{\tilde\pi}\otimes\tilde\pi^*(\can_\iota\otimes L\restr{Z})\big)
\tooo{2}^{\tilde\pi_*} \W^{*-c+1}(Z,\can_\iota\otimes L\restr{Z}).
$$
\end{enumerate}
In each case, the isomorphism~$\cong$ comes from square-periodicity in the twist~\eqref{square_eq} and the subsequent homomorphism is the push-forward~\eqref{push-forward_eq}.
\end{prop}

\begin{proof}
We only have to check the congruences in $\Pic/2$. By~\eqref{can_eq}, when $\ell\equiv c-1\!\mod2$, we have $[\can_\pi\otimes\pi^*L]=[(L,\ell)]=[M]$ in $\Pic(\Bl)/2$. When $\ell\equiv c\!\mod 2$, we have $[\can_{\tilde\pi}\otimes\tilde\pi^*(\can_\iota\otimes L\restr{Z})]=[(L\restr{Z},\ell)]=[M\restr{\Exc}]$ in $\Pic(\Exc)/2$. To apply~\eqref{push-forward_eq}, note that $\Kdim{\pi}=0$ since $\pi$ is birational and $\Kdim{\tilde\pi}=c-1$ since $\Exc=\bbP_Z(C_{Z/X})$ is the projective bundle of the rank-$c$ conormal bundle $C_{Z/X}$ over~$Z$.
\end{proof}

So far, we have only used Setup~\ref{setup}. Now add Hypothesis~\ref{basic_hypo} with $Y$ regular.

\begin{rem}
\label{key_rem}%
Since Picard groups of regular schemes are homotopy invariant, the $\bbA^{\!*}$-bundle $\alpha:U\to Y$ yields an isomorphism $\alpha^*:\Pic(Y)\isoto \Pic(U)$. Let us identify $\Pic(Y)$ with $\Pic(U)$, and hence with $\Pic(X)$  as we did above since $c=\codim_X(Z)\geq2$. We also have $\Line(\Exc)\restr{U}\simeq\cO_{U}$. Putting all this together, the right-hand part of Diagram~\eqref{hypo_eq} yields the following on Picard groups\,:
$$
\Pic\left(\vcenter{\xymatrix{
X
& \ U\ \ar@{_(->}[l]_-{\upsilon} \ar[d]^-{\alpha} \ar@{_(->}[ld]_-{\tilde\upsilon}
\\
\Bl \ar[u]^-{\pi} \ar[r]_-{\tilde\alpha}
& Y
}}\right)
\ \cong\
\vcenter{\xymatrix{
\Pic(X) \ar@{=}[r] \ar[d]_-{\smallmatrice{1\\0}}
& \Pic(X) \ar@{=}[d]
\\
\Pic(X)\oplus\bbZ \ar[ru]^(.35){(1\;0)\kern-.5em}
& \Pic(X)\,.\! \ar[l]_(.4){\smallmatrice{1\\\lambda}}
}}
$$
Note that the lower right map $\Pic(X)\cong\Pic(Y)\too\limits^{\tilde\alpha^*}\Pic(\Bl)\cong\Pic(X)\oplus \bbZ$ must be of the form $\smallmatrice{1\\\lambda}$ by commutativity (\ie since $\smallmatrice{1&0}\cdot\smallmatrice{1\\\lambda}=1$) but there is no reason for its second component $\lambda:\Pic(X)\to \bbZ$ to vanish. This is indeed a key observation. In other words, we have two homomorphisms from $\Pic(X)$ to $\Pic(\Bl)$, the direct one $\pi^*$ and the circumvolant one $\tilde\alpha^*\circ(\alpha^*)\inv\circ\upsilon^*$ going via $U$ and~$Y$
\begin{equation}
\label{non_comm_eq}%
\vcenter{\xymatrix{
\Pic(X) \ar[r]^-{\upsilon^*}_{\simeq} \ar[d]_-{\pi^*} \ar@{}[rd]|-{\not=}
& \Pic(U)  \ar[d]^-{(\alpha^*)\inv}_-{\simeq}
\\
\Pic(\Bl)
& \Pic(Y) \ar[l]^-{\tilde\alpha^*}
}}
\end{equation}
and they do \emph{not} coincide in general. The difference is measured by $\lambda$, which depends on the choice of $Y$ and on the choice of $\tilde\alpha:\Bl\to Y$, in Hypothesis~\ref{basic_hypo}.

So, for every $L\in\Pic(X)$, the integer~$\lambda(L)\in\bbZ$ is defined by the equation
\begin{equation}
\label{lambda_eq}%
\tilde\alpha^*\,(\alpha^*)\inv\,\upsilon^*(L)=\pi^*(L)\otimes \Line(\Exc)^{\otimes \lambda(L)}
\end{equation}
in $\Pic(\Bl)$. Under the isomorphism $\Pic(\Bl)\cong\Pic(X)\oplus\bbZ$, the above equation can be reformulated as $\tilde\alpha^*\,(\alpha^*)\inv\,\upsilon^*(L)=\big(L,\lambda(L)\big)$.
\end{rem}

\begin{thm}[Partial analogue of Theorem~\ref{standard_thm}]
\label{extend_thm}%
With the notation of~\ref{setup}, assume Hypothesis~\ref{basic_hypo} and assume $X,Y,Z$ regular. Recall that $c=\codim_X(Z)$. Let $L\in\Pic(X)$ and consider the integer $\lambda(L)\in\bbZ$ defined in~\eqref{lambda_eq} above.

If $\lambda(L)\equiv c-1 \!\mod 2$ then the restriction $\upsilon^*:\W^*(X,L)\to \W^*(U,L\restr{U})$ is split surjective, with an explicit section given by the composition $\pi_*\circ\tilde\alpha^*\circ(\alpha^*)\inv$
$$
\xymatrix{
\kern-10em \W^*(X,L)
& \W^*(U,L\restr{U}) \ar[d]^-{(\alpha^*)\inv}_-{\simeq}
\\
\W^*(\Bl\,,\,\can_\pi\otimes\pi^*L) \ar@<5em>[u]^-{\pi_*}
\cong
\W^*\big(\Bl\,,\,\tilde\alpha^*\,(\alpha^*)\inv L\restr{U}\big)
& \W^*\big(Y\,,\,(\alpha^*)\inv\,L\restr{U}\big) \ar[l]^-{\tilde\alpha^*}
}
$$
\end{thm}

\begin{proof}
The whole point is that $\pi_*$ can be applied after $\tilde\alpha^*\circ(\alpha^*)\inv$, that is, on $\W^*\big(\Bl\,,\,\tilde\alpha^*\,(\alpha^*)\inv\,\upsilon^*(L)\big)$. This holds by Proposition~\ref{dichotomy_prop}~(A) applied to
\begin{equation}
\label{M_eq}%
M:=\tilde\alpha^*\,(\alpha^*)\inv\,\upsilon^*(L)\equalby{=}{{\eqref{lambda_eq}}}(L,\lambda(L))\in\Pic(X)\oplus\bbZ=\Pic(\Bl)\,.
\end{equation}
The assumption~$\lambda(L)\equiv c-1\!\mod2$ expresses the hypothesis of Proposition~\ref{dichotomy_prop}~(A).
Checking that we indeed have a section goes as in the oriented case, see Thm.\,\ref{standard_thm}\,:
$$
\upsilon^*\circ\pi_*\circ\tilde\alpha^*\circ(\alpha^*)\inv
=\tilde\upsilon^*\circ\tilde\alpha^*\circ(\alpha^*)\inv
=\alpha^*\circ(\alpha^*)\inv
=\id.
$$
The first equality uses base-change~\cite[Thm.~6.9]{Calmes08b_pre}
on the left-hand cartesian square\,:
$$
\xymatrix@R=2.5em{
X
& \ U\ \ar@{_(->}[l]_-{\upsilon}
\\
\Bl \ar[u]^-{\pi}
& \ U\ \ar@{_(->}[l]_-{\tilde\upsilon} \ar[u]_-{\id}
}\kern7em
\xymatrix@R=2.1em{
L & L\restr{U}
\\
\can_\pi\otimes \pi^*L & L\restr{U}
}
$$
with respect to the right-hand line bundles. Note that $(\can_\pi)\restr{U}=\cO_U$ by~\eqref{can_eq}.
\end{proof}

\begin{rem}
In the above proof, see~\eqref{M_eq}, we do not apply Proposition~\ref{dichotomy_prop} to $M$ being $\pi^*L$, as one could first expect; see Remark~\ref{key_rem}. Consequently, our condition on~$L$, namely $\lambda(L)\equiv c-1 \!\mod 2$, does not only depend on the codimension~$c$ of $Z$ in~$X$ but also involves (hidden in the definition of~$\lambda$) the particular choice of the auxiliary scheme~$Y$ and of the morphism~$\tilde\alpha:\Bl\to Y$ of Hypothesis~\ref{basic_hypo}.
\end{rem}

\begin{rem}
The legitimate question is now to decide what to do in the remaining
case, that is, when $\lambda(L)\equiv c\!\mod 2$. As announced, this
is the central goal of our paper (Thm.\,\ref{mainreg_thm} below).
So, let $L\in\Pic(X)$ be a twist such that push-forward
along~$\pi:\Bl\to X$ cannot be applied to define a section to the
restriction $\W^*(X,L)\to \W^*(U,L\restr{U})$ as above. Actually, we
can find examples of such line bundles for which this restriction is
simply not surjective (see~Ex.\,\ref{Grass_exa}). The natural
problem then becomes to compute the possibly non-zero connecting
homomorphism $\bord:\W^*(U,L\restr{U})\to \W^{*+1}_Z(X,L)$. Although
not absolutely necessary, it actually simplifies the formulation
of~Theorem~\ref{mainreg_thm} to use \emph{d\'evissage} from~\cite[\S
6]{Calmes06_pre}, \ie the fact that push-forward along a regular
closed immersion is an isomorphism
\begin{equation}
\label{devissage_eq}%
\iota_*:\W^{*-c}(Z,\can_\iota\otimes L\restr{Z})\isoto \W^*_Z(X,L)\,.
\end{equation}
Using this isomorphism, we can replace the Witt groups with supports by Witt groups of~$Z$ in the localization long exact sequence, and obtain a long exact sequence
\begin{equation}
\label{W_LES_eq}\notag \addtocounter{equation}{1}%
\kern-.4em\xymatrix@C=.6em@R=.3em{
{\cdots} \ar[r]
& \W^*(X,L) \ar[rr]^-{\upsilon^*}
&& \W^*(U,L\restr{U}) \ar[rrdd]_-{\bord} \ar[rr]^-{\bord'}
&& \W^{*+1-c}(Z,\can_\iota\otimes L\restr{Z}) \ar[rr] \ar[dd]_{\simeq}^-{\iota_*}
&& \W^{*+1}(X,L) \ar[r]
& {\cdots}
\\
{\eqref{W_LES_eq}\!}
\\&&&&& \W^{*+1}_Z(X,L) \ar[rruu]
}
\end{equation}
We now want to describe~$\bord'$ when $\lambda(L)\equiv c\!\mod2$ (otherwise $\bord'=0$ by Thm.\,\ref{extend_thm}).

By the complete dichotomy of Proposition~\ref{dichotomy_prop}, we know that when the push-forward $\pi_*:\W^*(\Bl,M)\to \W^*(X,?)$ does not exist, here for $M=\tilde\alpha^*\,(\alpha^*)\inv\,\upsilon^*(L)$ by~\eqref{M_eq}, then the following composition $\tilde\pi_*\circ\tilde\iota^*$ exists and starts from the very group where~$\pi_*$ cannot be defined and arrives in the very group where $\bord'$ itself arrives\,:
$$
\W^*(\Bl,M)\too^{\tilde\iota^*} \W^*(\Exc,M\restr{\Exc})\too\limits^{\tilde\pi_*} \W^{*-c+1}(Z,\can_\iota\otimes L\restr{Z})\,.
$$
Hence, in a moment of exaltation, if we blindly apply this observation at the precise point where the oriented technique fails for Witt groups, we see that when we cannot define a section to restriction by the formula $\pi_*\circ\big(\tilde\alpha^*\circ(\alpha^*)\inv\big)$ we can instead define a mysterious homomorphism $(\tilde\pi_*\circ \tilde\iota^*)\circ\big(\tilde\alpha^*\circ(\alpha^*)\inv\big)$.
\end{rem}

\begin{mainthmreg}
\label{mainreg_thm}%
With the notation of~\ref{setup}, assume Hypothesis~\ref{basic_hypo} and assume $X,Y,Z$ regular. Let $L\in\Pic(X)$ and recall the integer $\lambda(L)\in\bbZ$ defined by~\eqref{lambda_eq}.

If $\lambda(L)\equiv c \!\mod 2$ then
the composition $\tilde\pi_*\circ\tilde\iota^*\circ\tilde\alpha^*\circ(\alpha^*)\inv$ is equal to the connecting homomorphism~$\bord'$ of~\eqref{W_LES_eq}, that is, the following diagram commutes\,:
$$
\!\xymatrix{
\kern-10em \W^{*+1-c}(Z,\can_\iota\otimes L\restr{Z})
& \kern0.5em \W^*(U,L\restr{U}) \kern0.5em \ar[d]^-{(\alpha^*)\inv}_-{\simeq} \ar[l]_-{\bord'}
\\
\W^*\big(\Exc\,,\,\can_{\tilde\pi}\otimes\tilde\pi^*(\can_\iota\otimes L\restr{Z})\big) \ar@<5em>[u]^-{\tilde\pi_*}
\cong
\W^*\big(\Exc\,,\,\tilde\iota^*\,\tilde\alpha^*\,(\alpha^*)\inv L\restr{U}\big)
& \W^*\big(Y,(\alpha^*)\inv\,L\restr{U}\big) \ar[l]^-{\tilde\iota^*\,\tilde\alpha^*}
}
$$
\end{mainthmreg}

This statement implies Thm.\,\ref{super_thm}\,(B) since $\bord=\iota_*\,\bord'$ by~\eqref{W_LES_eq}. Its proof will be given after generalization to the non-regular setting, at the end of Section~\ref{main_sec}.

\goodbreak

\begin{rem}
\label{punch_rem}%
Let us stress the peculiar combination of Theorem~\ref{extend_thm} and Theorem~\ref{mainreg_thm}. Start with a Witt class $w_U$ over the open $U\subset X$, for the duality twisted by some $L\in\Pic(U)=\Pic(X)$, and try to extend $w_U$ to a Witt class $w_X$ over~$X$\,:
$$
\xymatrix@C=6em{
\bord'(w_U) \ar@{}[rd]|*+[F]{\scriptstyle\leftarrow{\textrm{50\%\big|50\%}}\to}
& w_X \ar@{|.>}^-{\upsilon^*}[r]
& w_U \ar@{|->}[d]^-{(\alpha^*)\inv}
\\
w_{\Exc} \ar@{|.>}[u]^-{\displaystyle \tilde\pi_*}
& w_{\Bl} \ar@{|->}[l]^-{\tilde\iota^*} \ar@{|.>}[u]_-{\displaystyle \pi_*}
& w_Y \ar@{|->}[l]^-{\tilde\alpha^*}
}
$$
Then, \emph{either} we can apply the same construction as for
oriented theories, \ie push-forward the class
$w_{\Bl}:=\tilde\alpha^*\circ(\alpha^*)\inv(w_U)$ from $\Bl$ to $X$
along~$\pi$, constructing in this way an
extension~$w_X:=\pi_*(w_{\Bl})$ as wanted, \emph{or} this last
push-forward $\pi_*$ is forbidden on $w_{\Bl}$ because of the twist,
in which case the Witt class $w_U$ might simply not belong to the
image of restriction~$\upsilon^*$. The latter means that $w_U$ might
have a non-zero boundary $\bord'(w_U)$ over~$Z$, which then deserves
to be computed. The little miracle precisely is that in order to
compute this $\bord'(w_U)$, it suffices to resume the above process
where it failed, \ie with $w_{\Bl}$\,, and, since we cannot push it
forward along~$\pi$, we can consider the bifurcation of
Proposition~\ref{dichotomy_prop} and restrict this class $w_{\Bl}$
to the exceptional fiber~$\Exc$, say
$w_{\Exc}:=\tilde\iota^*w_{\Bl}$\,, and then push it forward
along~$\tilde\pi$. Of course, this does not construct an extension
of~$w_U$ anymore, since this new class $\tilde\pi_*(w_{\Exc})$ lives
over~$Z$, not over~$X$. Indeed, there is no reason a priori for this
new class to give anything sensible at all. Our Main Theorem is that
this construction in fact gives a formula for the
boundary~$\bord'(w_U)$.

\smallbreak

\noindent\textbf{Bottom line\,:} \emph{Essentially the same geometric recipe of pull-back and push-forward either splits the restriction or constructs the connecting homomorphism. In particular, the connecting homomorphism is explicitly described in both cases.}

\end{rem}

\section{The non-regular case}
\label{sing_sec}%

In Section~\ref{reg_sec}, we restricted our attention to the regular case in order to grasp the main ideas. However, most results can be stated in the greater generality of separated and noetherian $\bbZ[\frac{1}{2}]$-schemes admitting a dualizing complex. The goal of this section is to provide the relevant background and to extend Theorem~\ref{super_thm} to this non-regular setting, see Main Lemma~\ref{main_lem}.

\begin{rem}
\label{Wcoh_rem}%
The \emph{coherent Witt groups} $\Wcoh^*(X,K_X)$ of a scheme
$X\in\Sch$ (see~\ref{setup}) are defined using the derived category
$\Dbc(X)$ of complexes of $\cO_X$-modules whose cohomology is
coherent and bounded. Since $X$ is noetherian and separated,
$\Dbc(X)$ is equivalent to its subcategory $\Db(\coh(X))$ of bounded
complexes of coherent $\cO_X$-modules; see for
instance~\cite[Prop.\,A.4]{Calmes08b_pre}. The duality is defined
using the derived functor $\RHom(-,K_X)$ where $K_X\in\Dbc(X)$ is a
\emph{dualizing complex} (see~\cite[\S\,3]{Neeman08_pre}
or~\cite[\S\,2]{Calmes08b_pre}), meaning that the functor
$\RHom(-,K_X)$ defines a duality on~$\Dbc(X)$. For example, a scheme
is Gorenstein if and only if $\cO_X$ itself is an injectively
bounded dualizing complex and, in that case, all other dualizing
complexes are shifted line bundles (see Lemma~\ref{dualizing_lem}).
Regular schemes are Gorenstein, and for them, coherent Witt groups
coincide with the usual ``locally free'' Witt groups $\W^*(X,L)$
(\ie the ones defined using bounded complexes of locally free
sheaves instead of coherent ones). For any line bundle~$L$, we still
have a square-periodicity isomorphism
\begin{equation}
\label{re-square_eq}%
\Wcoh(X,K_X)\cong \Wcoh(X,K_X \otimes L^{\otimes 2})
\end{equation}
given by the multiplication by the class in $\W^0(X,L^{\otimes2})$ of the canonical form $L \to L^\vee \otimes L^{\otimes 2}$, using the pairing between locally free and coherent Witt groups.

For any closed embedding $Z \hookrightarrow X$ with open complement
$\upsilon:U\hookrightarrow X$, the restriction $K_U:=\upsilon^* K_X$
is a dualizing complex~\cite[Thm.~3.12]{Neeman08_pre} and the
general triangulated framework of~\cite{Balmer00} gives a
localization long exact sequence
\begin{equation} \label{LESLocCoh_eq}
\cdots \too^{\bord} \Wcoh^*_Z(X,K_X) \too \Wcoh^*(X,K_X) \too \Wcoh^*(U,K_U) \too^{\bord} \Wcoh^{*+1}_Z(X,K_X) \too \cdots\kern-.75em
\end{equation}
As for $K$-theory, no such sequence holds in general for singular schemes and locally free Witt groups.
\end{rem}

\begin{rem}
\label{push_rem}%
For coherent Witt groups, the push-forward along a proper morphism
$f:X' \to X$ takes the following very round form\,: If $K_X$ is a
dualizing complex on~$X$ then $f^!K_X$ is a dualizing complex
on~$X'$ (\cite[Prop.\,3.9]{Calmes08b_pre}) and the functor $\RR
f_*:\Dbc(X')\to \Dbc(X)$ induces a \emph{push-forward}
(\cite[Thm.\,4.4]{Calmes08b_pre})
\begin{equation}
\label{push_eq}%
f_*:\Wcoh^i(X',f^! K_X) \to \Wcoh^i(X,K_X)\,.
\end{equation}
Recall that $f^!:\Der_{\Qcoh}(X)\to \Der_{\Qcoh}(X')$ is the right
adjoint of~$\RR f_*$\,. If we twist the chosen dualizing complex
$K_X$ by a line bundle $L\in\Pic(X)$, this is transported to $X'$
via the following formula (see~\cite[Thm.\,3.7]{Calmes08b_pre})
\begin{equation}
\label{KL_eq}%
f^!(K_X \otimes L) \simeq f^!(K_X)\otimes f^* L \,.
\end{equation}
In the regular case, push-forward maps are also described in
Nenashev~\cite{Nenashev09}.
\end{rem}

\begin{rem}
\label{Tor_rem}%
Let us also recall from~\cite[Thm.\,4.1]{Calmes08b_pre} that the
pull-back
$$
f^*:\Wcoh^i(X, K_X) \to \Wcoh^i(X',\LL f^* K_X)
$$
along a finite Tor-dimension morphism $f:X'\to X$ is defined if $\LL
f^* (K_X)$ is a dualizing complex (this is not automatically true).
Together with the push-forward, this pull-back satisfies the usual
flat base-change formula (see~\cite[Thm.\,5.5]{Calmes08b_pre}).

A regular immersion $f:X'\hookrightarrow X$ has finite Tor-dimension since it is even perfect (see~\cite[p.\,250]{SGA6}). Moreover, in that case, $\LL f^*$ is the same as $f^!$ up to a twist and a shift (see Proposition~\ref{f^!_prop}), hence it preserves dualizing complexes.
\end{rem}

\begin{prop}
\label{dichotomygeneral_prop}%
In Setup~\ref{setup}, let $K_X$ be a dualizing complex on~$X$. Let $L\in\Pic(X)$ and $\ell\in\bbZ$. Then $K=\pi^! (K_X) \otimes \pi^*L \otimes \Line(E)^{\otimes \ell}$ is a dualizing complex on~$\Bl$ and any dualizing complex has this form, for some $L\in\Pic(X)$ and $\ell\in\bbZ$. Moreover, the dichotomy of Proposition~\ref{dichotomy_prop} here becomes\,:
\begin{enumerate}
\item[(A)]
If $\ell\equiv 0\!\mod 2$, we can push-forward along~$\pi:\Bl\to X$, as follows\,:
$$
\Wcoh^*(\Bl,K)\cong \Wcoh^*\big(\Bl,\pi^!(K_X\otimes L))\tooo{2}^{\pi_*} \Wcoh^*\big(X,K_X \otimes L)\,.
$$
\item[(B)]
If $\ell\equiv 1\!\mod 2$, we can push-forward along~$\tilde\pi:\Exc\to Z$, as follows\,:
$$
\Wcoh^*(\Exc,\LL \tilde\iota^* K)\cong \Wcoh^{*+1}\big(\Exc,\tilde\pi^!\iota^!(K_X \otimes L))\big)
\tooo{2}^{\tilde\pi_*} \Wcoh^{*+1}\big(Z,\iota^!(K_X\otimes L)\big).
$$
\end{enumerate}
As before, in both cases, the first isomorphism $\cong$ comes from square-periodicity~\eqref{re-square_eq} and the second morphism is push-forward~\eqref{push_eq}.
\end{prop}

\begin{proof}
The complex $K_{\Bl}:=\pi^! K_X$ is a dualizing complex on~$\Bl$ by Remark~\ref{push_rem}. By Lemma~\ref{dualizing_lem} and Proposition~\ref{Pic_codim2_prop}~(i), all dualizing complexes on $\Bl$ are of the form $K=\pi^! (K_X) \otimes \pi^*L \otimes \Line(E)^{\otimes \ell}$, for unique $L\in\Pic(X)$ and $\ell\in\bbZ$.

We only need to check the relevant parity for applying~\eqref{re-square_eq}. Case~(A) follows easily from~\eqref{KL_eq} by definition of $K$ and parity of~$\ell$. In~(B), we need to compare $\LL \tilde\iota^* K$ and $\tilde\pi^! \iota^! (K_X \otimes L)[1]$. By Proposition~\ref{can_prop}~(iv), we know that $\tilde\iota^!(-)\cong \tilde\iota^*\Line(E)[-1]\otimes \LL\tilde\iota^*(-)$. We apply this and~\eqref{KL_eq} in the second equality below, the first one using simply that $\iota\tilde\pi=\pi\tilde\iota$\,:
\begin{align*}
\tilde\pi^! \iota^!(K_X \otimes L)[1] & \cong
\tilde\iota^!\pi^!(K_X\otimes L)[1] \cong
\tilde\iota^*\Line(E)[-1]\otimes\LL\tilde\iota^*\big(\pi^!(K_X)\otimes
\pi^* L\big)[1] \cong
\\
& \cong \tilde\iota^*\Line(E)\otimes \LL\tilde\iota^*(K\otimes \Line(E)^{\otimes
-\ell}) \cong \tilde\iota^* \Line(E)^{\otimes (1-\ell)} \otimes \LL
\tilde\iota^* K.
\end{align*}
Since $1-\ell$ is even, $\tilde\iota^* \Line(E)^{\otimes (1-\ell)}$ is
a square, as desired.
\end{proof}

We now want to give the key technical result of the paper, which is an analogue of Theorem~\ref{super_thm} in the non-regular setting. The idea is to describe the connecting homomorphism on Witt classes over~$U$ which admit an extension to the blow-up~$\Bl$. The key fact is the existence of an additional twist on~$\Bl$, namely the twist by $\Line(E)$, which disappears on~$U$ (see~\ref{O(E)_def}) and hence allows Case~(B) below.

\begin{mainlem}
\label{main_lem}
In Setup~\ref{setup}, assume that $X$ has a dualizing complex~$K_X$ and let $K_U=\upsilon^*(K_X)$ and $K_{\Bl}=\pi^!(K_X)$; see Remarks~\ref{Wcoh_rem} and~\ref{push_rem}. Let $i\in\bbZ$.
\begin{enumerate}
\item[(A)] The following composition vanishes\,:
$$
\xymatrix{\Wcoh^i(\Bl,K_{\Bl}) \ar[r]^-{\tilde\upsilon^*}
& \Wcoh^i(U,K_U) \ar[r]^-{\bord} & \Wcoh^{i+1}_Z(X,K_X)\,.
}
$$
\smallbreak
\item[(B)] The following composition (well-defined since $\tilde\upsilon^*\Line(E)\simeq\cO_U$)
%
$$
\xymatrix@C=4ex{\Wcoh^{i}\big(\Bl,K_{\Bl} \otimes \Line(E)\big)
 \ar[r]^-{\tilde\upsilon^*}
& \Wcoh^i\big(U,K_U \otimes \tilde\upsilon^*\Line(E)\big) \cong \Wcoh^i(U,K_U)
 \ar[r]^-{\bord}
& \Wcoh^{i+1}_Z(X,K_X)
}
$$
coincides with the composition
$$
\xymatrix{
\Wcoh^{i}(\Bl,\Line(E) \otimes K_{\Bl})
 \ar[d]_{\tilde\iota^*}
&& \Wcoh^{i+1}_Z(X, K_X)
\\
\Wcoh^i\big(E, \LL\tilde\iota^*(\Line(E) \otimes K_{\Bl})\big)
 \ar@{}[r]|-{\cong}
& \Wcoh^{i+1}\big(E,\tilde\pi^!\iota^!K_X\big)
 \ar[r]^-{\tilde\pi_*}
& \Wcoh^{i+1}\big(Z,\iota^!K_X\big)
 \ar[u]_{\iota_*}
}$$
where the latter isomorphism $\cong$ is induced by the composition
\begin{equation}
\label{isos1_eq}%
\LL\tilde\iota^*(\Line(E) \otimes K_{\Bl}) \cong \tilde\iota^*(\Line(E)) \otimes \LL\tilde\iota^*(K_{\Bl}) \cong \tilde\iota^!K_{\Bl}[1] \cong \tilde\pi^!\iota^!K_X[1]\,.
\end{equation}
\end{enumerate}%
\end{mainlem}

The proof of this result occupies Section~\ref{codim1_sec}. Here are just a couple of comments on the statement. Let us first of all explain the announced sequence of isomorphisms~\eqref{isos1_eq}. The first one holds since $\LL\tilde\iota^*$ is a tensor functor and since $\Line(E)$ is a line bundle (hence is flat). The second one holds by Proposition~\ref{concrete_can_prop}~(v). The last one follows by definition of $K_{\Bl}$ and the fact that $\iota\,\tilde\pi=\pi\,\tilde\iota$. Finally, note that we use the pull-back $\tilde\iota^*$ on coherent Witt groups as recalled in Remark~\ref{Tor_rem}.

\section{The main argument}
\label{codim1_sec}%

Surprisingly enough for a problem involving the blow-up
$\Bl=\Bl_Z(X)$ of $X$ along~$Z$, see~\eqref{blowup_eq}, the case
where $\codim_X(Z)=1$ is also interesting, even though, of
course, in that case $\Bl=X$ and $\Exc=Z$. In fact, this case is
crucial for the proof of Main Lemma~\ref{main_lem} and this is why
we deal with it first. In the ``general'' proof where $\codim_X(Z)$
is arbitrary, we will apply the case of codimension one to
$\tilde\iota:\Exc\hookrightarrow\Bl$. Therefore, we use the
following notation to discuss codimension one.

\begin{nota}
\label{codim1_nota}%
Let $B\in \Sch$ be a scheme with a dualizing complex $K_B$ and $\tilde\iota:E\hookrightarrow B$ be a prime divisor, that is, a regular closed immersion of codimension one, of a subscheme $E \in \Sch$. Let $\Line(E)$ be the line bundle on $B$ associated to~$E$ (see Definition~\ref{O(E)_def}). Let $\tilde\upsilon:U \hookrightarrow B$ be the open immersion of the open complement
$$
\xymatrix{E\ \ar@{^(->}[r]^-{\tilde\iota} & B & \ U \ar@{_(->}[l]_-{\tilde\upsilon}}
$$
$U=B-E$ and let $K_U$ be the dualizing complex $\tilde\upsilon^* (K_B)$ on~$U$.
\end{nota}

\begin{mainlemcodimone}
\label{mainCodimOne_lem}
With Notation~\ref{codim1_nota}, let $i\in\bbZ$. Then\,:
\begin{enumerate}
\item[(A)] The composition
$$
\xymatrix{\Wcoh^i(B,K_B) \ar[r]^-{\tilde\upsilon^*}
& \Wcoh^i(U,K_U) \ar[r]^-{\bord}
& \Wcoh^{i+1}_E(B,K_B)
}$$
is zero.
\item[(B)] The composition
$$
\xymatrix@C=4.5ex{
\Wcoh^{i}\big(B,K_B \otimes \Line(E)\big) \ar[r]^-{\tilde\upsilon^*}
& \Wcoh^i\big(U,K_U \otimes \tilde\upsilon^*\Line(E)\big) \cong \Wcoh^i(U,K_U)
 \ar[r]^-{\bord}
& \Wcoh^{i+1}_E(B,K_B)
}$$
coincides with the composition
$$
\xymatrix{
\Wcoh^{i}(B,\Line(E) \otimes K_B)
 \ar[d]_{\tilde\iota^*}
&& \Wcoh^{i+1}_E(B, K_B)
\\
\Wcoh^i\big(E, \LL\tilde\iota^*(\Line(E) \otimes K_B)\big)
 \ar@{}[r]|-{\cong}
& \Wcoh^i\big(E, \tilde\iota^!K_B[1]\big)
 \ar@{}[r]|-{\cong}
& \Wcoh^{i+1}\big(E, \tilde\iota^!K_B\big)
 \ar[u]_{\tilde\iota_*}
}
$$
where the first isomorphism $\cong$ is induced by the following isomorphism
\begin{equation}
\label{isos2_eq}%
\LL\tilde\iota^*(\Line(E) \otimes K_B) \cong \tilde\iota^*(\Line(E)) \otimes \LL\tilde\iota^*(K_B) \cong \tilde\iota^!(K_B)[1].
\end{equation}
\end{enumerate}
\end{mainlemcodimone}

\begin{proof}
Case (A) is simple\,: The composition of two consecutive morphisms
in the localization long exact sequence~\eqref{LESLocCoh_eq} is
zero. Case (B) is the nontrivial one. The isomorphisms~\eqref{isos2_eq} are the same as in~\eqref{isos1_eq}.

At this stage, we upload the definition of the connecting homomorphism for Witt groups $\bord:\Wcoh^i(U,K_U)\to \Wcoh^{i+1}_E(B,K_B)$, which goes as follows\,: Take a non-degenerate symmetric space $(P,\phi)$ over~$U$ for the $i^{\textrm{th}}$-shifted duality with values in $K_U$\,; there exists a possibly degenerate symmetric pair $(Q,\psi)$ over $B$ for the same duality (with values in $K_B$) which restricts to $(P,\phi)$ over~$U$\,; compute its symmetric cone $d(Q,\psi)$, which is essentially the cone of $\psi$ equipped with a natural metabolic form; see~\cite[\S\,4]{Balmer00} or~\cite[Def.\,2.3]{Balmer05b} for instance\,; for any choice of such a pair $(Q,\psi)$, the boundary $\bord(P,\phi)\in \Wcoh^{i+1}_E(B,K_B)$ is the Witt class of $d(Q,\psi)$.

There is nothing really specific to dualizing complexes here. The above construction is a purely triangular one, as long as one uses the \emph{same} duality for the ambient scheme $B$, for the open $U\subset B$ and for the Witt group of $B$ with supports in the closed complement~$E$. The subtlety of statement~(B) is that we start with a twisted duality on the scheme $B$ which is not the duality used for~$\bord$, but which agrees with it on~$U$ by the first isomorphism~$\cong$ in statement~(B).

Now, take an element in $\Wcoh^i(B,\Line(E) \otimes K_B)$. It is the Witt-equivalence class of a symmetric space $(P,\phi)$ over $B$ with respect to the $i^{\textrm{th}}$-shifted duality with values in~$\Line(E)\otimes K_B$.
The claim of the statement is that, modulo the above identifications of dualizing complexes, we should have
\begin{equation}
\label{claim_eq}%
\bord(\tilde\upsilon^*(P,\phi))=\tilde\iota_*(\tilde\iota^*(P,\phi))
\end{equation}
in $\Wcoh^{i+1}_E(B,K_B)$. By the above discussion, in order to compute $\bord(\tilde\upsilon^*(P,\phi))$, we need to find a symmetric pair $(Q,\psi)$ over~$B$, for the duality given by $K_B$, and such that $\tilde\upsilon^*(Q,\psi)=\tilde\upsilon^*(P,\phi)$. Note that we cannot take for $(Q,\psi)$ the pair $(P,\phi)$ itself because $(P,\phi)$ is symmetric for the twisted duality $\Line(E)\otimes K_B$ on~$B$. Nevertheless, it is easy to ``correct'' $(P,\phi)$ as follows.

As in Definition~\ref{O(E)_def}, we have a canonical homomorphism of line bundles\,:
$$
\sigma_{\!E}:\Line(E)^\vee\to \cO_{B}\,.
$$
The pair $(\Line(E)^\vee,\sigma_{\!E})$ is symmetric in the derived category $\Db(\VB(B))$ of vector bundles over~$B$, with respect to the $0^{\textrm{th}}$-shifted duality twisted by $\Line(E)^\vee$, because the target of $\sigma_{\!E}$ is the dual of its source\,: $(\Line(E)^\vee)^\vee[0]\otimes\Line(E)^\vee\cong\cO_{B}$.

Let us define the wanted symmetric pair $(Q,\psi)$ in~$\Dbc(B)$ for the $i^{\textrm{th}}$-shifted duality with values in $K_B$ as the following product\,:
$$
(Q,\psi):=(\Line(E)^\vee,\sigma_{\!E})\otimes(P,\phi)\,.
$$
Note that we tensor a complex of vector bundles with a coherent one to get a coherent one, following the formalism of~\cite[\S\,4]{Balmer05b} where such external products are denoted by~$\star$. We claim that the restriction of $(Q,\psi)$ to~$U$ is nothing but $\tilde\upsilon^*(P,\phi)$. This is easy to check since $\Line(E)\restr{U}=\cO_{U}$ via $\sigma_{\!E}$ (see~\ref{O(E)_def}), which means $(\Line(E)^\vee,\sigma_{\!E})\restr{U}=1_{U}$. So, by the construction of the connecting homomorphism $\bord$ recalled at the beginning of the proof, we know that $\bord(\tilde\upsilon^*(P,\phi))$ can be computed as $d(Q,\psi)$. This reads\,:
$$
\bord(\tilde\upsilon^*(P,\phi))=d\big((\Line(E)^\vee,\sigma_{\!E})\otimes(P,\phi)\big).
$$
Now, we use that $(P,\phi)$ is non-degenerate and that therefore (see~\cite[Rem.\,5.4]{Balmer05b} if necessary) we can take $(P,\phi)$ out of the above symmetric cone $d(...)$, \ie
\begin{equation}
\label{proof0_eq}%
\bord(\tilde\upsilon^*(P,\phi))=d(\Line(E)^\vee,\sigma_{\!E})\otimes(P,\phi)\,.
\end{equation}

Let us compute the symmetric cone $d(\Line(E)^\vee,\sigma_{\!E})=:(C,\chi)$. Note that this only involves vector bundles. We define $C$ to be the cone of $\sigma_{\!E}$ and we equip it with a symmetric form $\chi:C\isoto C^\vee[1]\otimes\Line(E)^\vee$ for the duality used for $(\Line(E),\sigma_{\!E})$ but shifted by one, that is, for the $1^{\textrm{st}}$ shifted duality with values in~$\Line(E)^\vee$. One checks that $(C,\chi)$ is given by the following explicit formula\,:
\begin{equation}
\label{chi_eq}%
\vcenter{\xymatrix@C=1.5em{ C=\ar[d]_-{\chi}\kern-1em
&
{\big(}{\cdots}\ar[r]
& 0 \ar[r] \ar[d]
& 0 \ar[r] \ar[d]
& \Line(E)^\vee \ar[r]^-{\sigma_{\!E}} \ar[d]_-{-1}
& \cO_B \ar[r] \ar[d]^-{1}
& 0 \ar[r] \ar[d]
& 0 \ar[r] \ar[d]
& {\cdots}\big)
\\
C^\vee[1]\otimes \Line(E)^\vee=\kern-2em
&
{\big(}{\cdots}\ar[r]
& 0 \ar[r]
& 0 \ar[r]
& \Line(E)^\vee \ar[r]_-{-(\sigma_{\!E})^\vee\kern0.8em\vphantom{T^{T^T}}}
& \cO_B \ar[r]
& 0 \ar[r]
& 0 \ar[r]
& {\cdots}\big)
}}
\end{equation}
where the complexes have $\cO_B$ in degree zero. Now, observe that
the complex $C$ is a resolution of $\tilde\iota_*(\cO_E)$ over~$B$,
by Definition~\ref{O(E)_def}, that is, $C\simeq\tilde\iota_*(\cO_E)$
in the derived category of~$B$. Moreover, by
Propositions~\ref{f^!_prop} and~\ref{concrete_can_prop}~(ii), we
have
$\tilde\iota^!(\Line(E)^\vee[1])=\can_{\tilde\iota}[-1]\otimes\tilde\iota^*(\Line(E)^\vee[1])\cong\cO_E$.
Using this, one checks the conceptually obvious fact that $\chi$ is
also the push-forward along the perfect morphism $\tilde\iota$ of
the unit form on~$\cO_E$. See Remark~\ref{pushone_rem} below for
more details.
This means that we have an isometry in $\Db_E(\VB(B))$
$$
d\big(\Line(E)^\vee,\sigma_{\!E}\big)=\tilde\iota_*(1_E)
$$
of symmetric spaces with respect to the $1^{\textrm{st}}$ shifted
duality with values in~$\Line(E)^\vee$. Plugging this last equality
in~\eqref{proof0_eq}, and using the projection formula (see
Remark~\ref{pushone_rem}) we obtain
$$
\bord(\tilde\upsilon^*(P,\phi))=\tilde\iota_*(1_E)\otimes(P,\phi)
=\tilde\iota_*\big(1_E\otimes\tilde\iota^*(P,\phi)\big)
=\tilde\iota_*\big(\tilde\iota^*(P,\phi)\big)\,.
$$
This is the claimed equality~\eqref{claim_eq}.
\end{proof}

\begin{rem} \label{pushone_rem}
In the above proof, we use the ``conceptually obvious fact'' that
the push-forward of the unit form on~$\cO_E$ is indeed the $\chi$
of~\eqref{chi_eq}. This fact is obvious to the expert but we cannot
provide a direct reference for this exact statement. However, if the
reader does not want to do this lengthy verification directly, the
computation of~\cite[\S\,7.2]{Calmes08b_pre} can be applied
essentially verbatim. The main difference is that here, we are
considering a push-forward of locally free instead of coherent Witt
groups along a regular embedding. Such a push-forward is constructed
using the same tensor formalism as the proper push-forwards for
coherent Witt groups considered in \loccit\ along morphisms that are
proper, perfect and Gorenstein, which is true of a regular
embedding. In \loccit\ there is an assumption that the schemes are
Gorenstein, ensuring that the line bundles are dualizing complexes.
But here, the dualizing objects for our category of complexes of
locally free sheaves are line bundles anyway and the extra
Gorenstein assumption is irrelevant.

Moreover, the projection formula used in the above proof is
established in complete generality for non necessarily regular
schemes by the same method as in~\cite[\S\,5.7]{Calmes08b_pre} using
the pairing between the locally free derived category and the
coherent one to the coherent one. More precisely, this pairing is
just a restriction of the quasi-coherent pairing
$\Der_{\Qcoh}\times\Der_{\Qcoh}\overset{\otimes}\too\Der_{\Qcoh}$ of
\loccit\ to these subcategories. By the general tensor formalism of
\cite{Calmes09}, for any morphism $f:X \to Y$ as above, for any
object $A$ (resp. $B$) in the quasi-coherent derived category of $X$
(resp. $Y$), we obtain a projection morphism in~$\Der_{\Qcoh}(Y)$
$$
\RR f_*(A) \otimes B \too \RR f_*(A \otimes \LL f^*(B)),
$$
see~\cite[Prop.\,4.2.5]{Calmes09}. It is an isomorphism
by~\cite[Thm.\,3.7]{Calmes08b_pre}. We actually only use it for $A$
a complex of locally free sheaves and $B$ a complex with coherent
and bounded cohomology. The projection formula is implied by
\cite[Thm.\,5.5.1]{Calmes09}.
\end{rem}

\begin{proof}[Proof of Main Lemma~\ref{main_lem}]
Case (A) follows from the codimension one case and the compatibility of push-forwards with connecting homomorphisms (here along the identity of~$U$). Case (B) follows from the outer commutativity of the following diagram\,:
\begin{equation}
\label{final_eq}%
\vcenter{\xymatrix{ \Wcoh^i(U,K_U)
 \ar[r]^-{\partial} \ar@{=}[d]
& \Wcoh^{i+1}_Z(X,K_X)
& \Wcoh^{i+1}(Z,\iota^!K_X)
 \ar[l]_-{\iota_*}
\\
\Wcoh^i(U,K_U) \ar[r]^-{\partial}
& \Wcoh^{i+1}_{\Exc}(\Bl,K_{\Bl})
 \ar[u]_{\pi_*}
& \Wcoh^{i+1}(\Exc, \tilde\pi^!\iota^!K_X)
 \ar[l]_-{\tilde\iota_*} \ar[u]_-{\tilde\pi_*}
\\
\Wcoh^i(\Bl,\Line(E)\otimes K_{\Bl}) \ar[u]^{\tilde\upsilon^*} \ar[r]^-{\tilde\iota^*}
& \Wcoh^i\big(\Bl,\LL\tilde\iota^*(\Line(E)\otimes K_{\Bl})\big)
 \ar@{}[r]|-{\cong} & \Wcoh^{i+1}(E,\tilde\iota^!K_{\Bl}) \ar[u]^-{\cong} \\
}}
\end{equation}
We shall now verify the inner commutativity of this diagram.
The upper left square of~\eqref{final_eq} commutes by compatibility of push-forward with connecting homomorphisms. The upper right square of~\eqref{final_eq} simply commutes by functoriality of push-forward applied to $\iota\circ\tilde\pi=\pi\circ\tilde\iota$.
Most interestingly, the lower part of~\eqref{final_eq} commutes by Lemma~\ref{mainCodimOne_lem} applied to the codimension one
inclusion $\tilde\iota: \Exc \hookrightarrow \Bl$.
\end{proof}

\section{The Main Theorem in the non-regular case}
\label{main_sec}%

Without regularity assumptions, we have shown in Main Lemma~\ref{main_lem} how to compute the connecting homomorphism $\bord:\Wcoh^*(U,K_U)\to \Wcoh^{*+1}_Z(X,K_X)$ on those Witt classes over~$U$ which come from $\Bl=\Bl_Z(X)$ by restriction $\tilde\upsilon^*$. The whole point of adding Hypothesis~\ref{basic_hypo} is precisely to split~$\tilde\upsilon^*$, that is, to construct for each Witt class on~$U$ an extension on~$\Bl$. In the regular case, this follows from homotopy invariance of Picard groups and Witt groups. In the non-regular setting, things are a little more complicated. Let us give the statement and comment on the hypotheses afterwards (see Remark~\ref{Gor_rem}).

\begin{mainthmsing} \label{main_thm}
In Setup~\ref{setup}, assume that $X$ has a dualizing complex~$K_X$ and equip $U$ with the restricted complex $K_U=\upsilon^*(K_X)$. Assume Hypothesis~\ref{basic_hypo} and further make the following hypotheses\,:
\begin{enumerate}
\item There exists a dualizing complex $K_Y$ on $Y$ such that $\alpha^* K_Y= K_U$.
\item The $\bbA^*$-bundle $\alpha$ induces an isomorphism $\Wcoh^*(Y,K_Y)\isoto \Wcoh^*(U,K_U)$.
\item The morphism $\tilde\alpha$ is of finite Tor dimension and $\LL \tilde\alpha^*(K_Y)$ is dualizing.
\item Sequence~\eqref{SESPic_eq} is exact\,: $\bbZ \to \Pic(\Bl) \to \Pic(U)$. (See Proposition~\ref{Pic_codim1_prop}.)
\end{enumerate}
Then $\LL \tilde\alpha^*(K_Y) \simeq \pi^! K_X \otimes
\Line(E)^{\otimes n}$ for some $n\in\bbZ$, and the following holds true\,:
\begin{itemize}
\item[(A)] If $n$ can be chosen even, the composition $\pi_* \tilde\alpha^* (\alpha^*)\inv$ is a section of $\upsilon^*$.
\item[(B)] If $n$ can be chosen odd, the composition $\iota_*\tilde\pi_*\tilde\iota^*\tilde\alpha^* (\alpha^*)\inv$ coincides with the connecting homomorphism $\bord:\Wcoh^*(U,K_U)\to \Wcoh^{*+1}_E(X,K_X)$.
\end{itemize}
\end{mainthmsing}

\begin{proof}
By~(c) and Remark~\ref{push_rem} respectively, both $\LL \tilde\alpha^*(K_Y)$ and $\pi^! K_X$ are dualizing complexes on~$\Bl$. By Lemma~\ref{dualizing_lem}~(i), they differ by a shifted line bundle\,:  $\LL \tilde\alpha^*(K_Y) \simeq \pi^! K_X \otimes L[m]$ with $L\in\Pic(\Bl)$ and $m\in\bbZ$. Restricting to~$U$, we get
$$
K_U \otimes \tilde\upsilon^* L[m] \simeq \tilde\upsilon^*\pi^! K_X \otimes \tilde\upsilon^*L[m] \simeq \tilde\upsilon^*(\pi^! K_X \otimes L[m] )\simeq \tilde\upsilon^* \LL \tilde\alpha^*(K_Y) \simeq \alpha^* K_Y \simeq K_U
$$
where the first equality holds by flat base-change (\cite[Thm.\,5.5]{Calmes08b_pre}). Thus, $\tilde\upsilon^* L[m]$ is the trivial line bundle on~$U$ by Lemma~\ref{dualizing_lem}~(ii). So $m=0$ and, by~(d), $L \simeq \Line(E)^{\otimes n}$ for some $n\in\bbZ$. This gives $\LL \tilde\alpha^*(K_Y) \simeq \pi^! K_X \otimes \Line(E)^{\otimes n}$ as claimed.

We now consider coherent Witt groups. By~(c) and Remark~\ref{Tor_rem}, $\tilde\alpha$ induces a morphism $\tilde\alpha^*:\Wcoh^*(Y,K_Y)\to\Wcoh(\Bl,\LL\tilde\alpha^*K_Y)$. By Lemma~\ref{affinebundle_lem}, the flat morphism $\alpha$ induces a homomorphism $\alpha^*:\Wcoh^*(Y,K_Y)\to \Wcoh^*(U,K_U)$ which is assumed to be an isomorphism in~(b). So, we can use $(\alpha^*)\inv$. When $n$ is even, we have
$$\upsilon^* \pi_* \tilde\alpha^* (\alpha^*)\inv = \tilde\upsilon^* \tilde\alpha^* (\alpha^*)\inv = \alpha^* (\alpha^*)\inv = \id
$$
where the first equality holds by flat base-change (\cite[Thm.\,5.5]{Calmes08b_pre}). This proves~(A). On the other hand, when $n$ is odd, we have
$$
\iota_* \tilde\pi_* \tilde\iota^*\tilde\alpha^*(\alpha^*)\inv = \bord\,\tilde\upsilon^* \tilde\alpha^* (\alpha^*)\inv = \bord\,\alpha^* (\alpha^*)\inv = \bord
$$
where the first equality holds by Main Lemma~\ref{main_lem}~(B).
\end{proof}

\begin{rem}
\label{Gor_rem}%
Hypothesis (a) in Theorem~\ref{main_thm} is always true when $Y$ admits a dualizing complex and homotopy invariance holds over $Y$ for the Picard group (\eg{} $Y$ regular).
Homotopy invariance for coherent Witt groups should hold in general but only appears in the literature when $Y$ is Gorenstein, see Gille~\cite{Gille03b}. This means that Hypothesis~(b) is a mild one. Hypothesis (d) is discussed in Proposition~\ref{Pic_codim1_prop}.
\end{rem}

\begin{rem}
\label{re-lambda_rem}%
In Theorem~\ref{main_thm}, the equation $\LL \tilde\alpha^*(K_Y) \simeq \pi^! K_X \otimes \Line(E)^{\otimes n}$, for $n\in\bbZ$, should be considered as a non-regular analogue of Equation~\eqref{lambda_eq}. In Remark~\ref{key_rem}, we discussed the compatibility of the various lines bundles on the schemes~$X$, $U$, $Y$ and~$\Bl$. Here, we need to control the relationship between dualizing complexes instead and we do so by restricting to~$U$ and by using the exact sequence~\eqref{SESPic_eq}. Alternatively, one can remove Hypothesis~(d) and directly assume the relation $\LL\tilde\alpha^*(K_Y)\simeq \pi^!(K_X)\otimes\Line(E)^{\otimes n}$ for some~$n\in\bbZ$. This might hold in some particular examples even if~\eqref{SESPic_eq} is not exact.
\end{rem}
For the convenience of the reader, we include the proofs of the following facts.

\begin{lem}
If $X$ is Gorenstein, then $Z$ and $\Bl$ are Gorenstein.
If $X$ is regular, $\Bl$ is regular.
\end{lem}
\begin{proof}
By Prop.\,\ref{f^!_prop}, $\pi^!(\cO_X)$ is the line bundle
$\can_\pi$. Since $\pi$ is proper, $\pi^!$ preserves injectively
bounded dualizing complexes and $\can_\pi$ is dualizing and since it
is a line bundle, $\Bl$ is Gorenstein. The same proof holds for $Z$,
since $\iota^!(\cO_X)$ is $\can_\iota$ (shifted) which is also a
line bundle by Prop.\,\ref{can_prop}. For regularity,
see~\cite[Thm.\,1.19]{Liu02}.
\end{proof}

\begin{proof}[Proof of Theorem~\ref{mainreg_thm}] Note that all the
assumptions of Theorem~\ref{main_thm} are fulfilled in the regular
case, that is, in the setting of Section~\ref{reg_sec}.
Indeed, if $X$ and $Y$ are regular, $\Bl$ and $U$ are regular, and
the dualizing complexes on $X$, $Y$, $\Bl$ and $U$ are simply
shifted line bundles. The morphism $\alpha^*: \Pic(Y) \to \Pic(U)$
is then an isomorphism (homotopy invariance) and $\tilde\alpha$ is
automatically of finite Tor dimension, as any morphism to a regular
scheme. Finally, the sequence on Picard groups is exact by
Proposition~\ref{Pic_codim1_prop}.

Let $K_X=L$ be the chosen line bundle on $X$. Then set $L_U:=K_U=\upsilon^* L$ and choose $L_Y=K_Y$ to be the unique line bundle (up to isomorphism) such that $\alpha^*L_Y = L_U$. By~\eqref{lambda_eq}, we have $\tilde\alpha^* L_Y = \pi^* L\otimes \Line(E)^{\otimes\lambda(L)} = \pi^! L \otimes \Line(E)^{\otimes(\lambda(L)-c+1)}$, where the last equality holds since $\pi^! L = \Line(E)^{\otimes (c-1)} \otimes \pi^* L$ by Proposition~\ref{can_prop}~(vi). In other words, we have proved that $\tilde\alpha^* K_Y = \pi^! K_X \otimes \Line(E)^{\otimes(\lambda(L)-c+1)}$. In Theorem~\ref{main_thm}, we can then take $n=\lambda(L)-c+1$ and the parity condition becomes $\lambda(L)\equiv c-1 \!\mod 2$ for Case (A) and $\lambda(L)\equiv c \! \mod 2$ for Case (B). So, Case (A) is the trivial one and corresponds to Theorem~\ref{extend_thm}. Case (B) exactly gives Theorem~\ref{mainreg_thm} up to the identifications of line bundles explained in Appendix~\ref{app}.
\end{proof}

\begin{appendix}
\section{Line bundles and dualizing complexes}
\label{app}

We use Hartshorne~\cite{Hartshorne77} or Liu~\cite{Liu02} as general
references for algebraic geometry. We still denote by $\Sch$ the
category of noetherian separated connected schemes (we do not need
``over $\bbZ[\frac12]$'' in this appendix).

\begin{defi}
\label{O(E)_def}%
Let $\tilde\iota: E \hookrightarrow B$ be a regular closed immersion
of codimension one, with $B\in Sch$.
Consider the ideal $I_E\subset\cO_B$ defining~$E$
\begin{equation}
\label{O(E)_eq}%
\xymatrix{0 \ar[r] & I_E \ar[r]^-{\sigma_{\!E}} & \cO_B \ar[r] & \tilde\iota_*\cO_E \ar[r] & 0}.
\end{equation}
By assumption, $I_E$ is an invertible ideal, \ie a line bundle. The
\emph{line bundle associated to~$E$} is defined as its dual
$\Line(E):=(I_E)^\vee$, see~\cite[II.6.18]{Hartshorne77}. We thus
have by construction a global section $\sigma_{\!E}:\Line(E)^\vee\to
\cO_B$, which vanishes exactly on~$E$. This gives an explicit trivialization
of $\Line(E)$ outside~$E$. On the other hand, the
restriction of $\Line(E)$ to $E$ is the normal bundle
$\Line(E)\restr{E}\cong N_{E/B}$.
\end{defi}

\begin{exa}
\label{L_Exc_exa}%
Let $\Bl=\Bl_Z(X)$ be the blow-up of $X$ along a regular closed
immersion $Z\hookrightarrow X$ as in Setup~\ref{setup}. Let
$I=I_Z\subset \cO_X$ be the sheaf of ideals defining~$Z$. By
construction of the blow-up, we have $\Bl=\Proj(\cS)$ where $\cS$ is the
sheaf of graded $\cO_X$-algebras
$$
\cS:=\cO_X\oplus I\oplus I^2\oplus I^3\oplus\cdots
$$
Similarly, $\Exc=\Proj(S/\cJ)$ where $\cJ:=I\cdot\cS\subset\cS$ is the sheaf of homogeneous ideals
$$
\cJ=I\oplus I^2\oplus I^3\oplus I^4\oplus \cdots
$$
So, $\Exc=\bbP_Z(C_{Z/X})$ is a projective bundle over~$Z$ associated to the vector bundle $C_{Z/X}=I/I^2$ which is the conormal bundle of~$Z$ in~$X$.
Associating $\cO_{\Bl}$-sheaves to graded $\cS$-modules, the obvious
exact sequence $0\to \cJ\to \cS\to \cS/\cJ\to 0$ yields
\begin{equation} \label{codim1ses_eq}
\xymatrix{0 \ar[r] & {\widetilde{\cJ}}_{\vphantom{_{j}}} \ar[r]^-{\sigma_{\!\Exc}} & \cO_{\Bl} \ar[r] & \tilde\iota_*\cO_{\Exc} \ar[r] & 0}.
\end{equation}
Compare~\eqref{O(E)_eq}. This means that here
$I_{\Exc}=\widetilde{\cJ}$. But now, $\cJ$ is obviously $\cS(1)$
truncated in non-negative degrees. Since two graded $\cS$-modules
which coincide above some degree have the same associated sheaves,
we have $I_{\Exc}=\widetilde{\cJ}=\widetilde{\cS(1)}=\cO_{\Bl}(1)$.
Consequently, $\Line(\Exc)=(I_{\Exc})^\vee=\cO_{\Bl}(-1)$. In
particular, we get
\begin{equation}
\label{restr_O(E)_eq}%
\Line(\Exc)\restr{\Exc}=\cO_{\Bl}(-1)\restr{\Exc}=\cO_{\Exc}(-1).
\end{equation}
\end{exa}

\begin{prop}[Picard group in codimension one] \label{Pic_codim1_prop}
Let $B\in \Sch$ be a scheme and $\tilde\iota:E\hookrightarrow B$ be
a regular closed immersion of codimension one of an irreducible subscheme $E\in \Sch$ with open complement
$\tilde\upsilon:U\hookrightarrow B$. We then have a complex
\begin{equation} \label{SESPic_eq}
\xymatrix{ \bbZ \ar[r] & \Pic(B) \ar[r]^-{\tilde\upsilon^*} & \Pic(U)}
\end{equation}
where the first map sends $1$ to the line bundle $\Line(E)$ associated to~$E$. This complex is exact if $B$ is normal, and $\tilde\upsilon^*$ is surjective when $B$ is furthermore regular. It is also exact when $B$ is the blow-up of a normal scheme $X$ along a regular embedding.
\end{prop}

\begin{proof}
\eqref{SESPic_eq} is a complex since $\Line(E)$ is trivial on $U$.
When $B$ is normal, $\Pic(B)$ injects in the group ${\rm Cl}(B)$ of
Weil divisors (see~\cite[7.1.19 and 7.2.14~(c)]{Liu02}), for which
the same sequence holds by~\cite[Prop.~II.6.5]{Hartshorne77}.
Exactness of~\eqref{SESPic_eq} then follows by diagram chase. The
surjectivity of $\tilde\upsilon^*$ when $B$ is regular follows from
\cite[Prop.~II.6.7 (c)]{Hartshorne77}. When $B$ is the blow-up of
$X$ along $Z$, we can assume that $\codim_X(Z)\geq2$ by the previous
point. Then, the result again follows by diagram chase, using that
$\Pic(B)=\Pic(X) \oplus \bbZ$, as proved in
Proposition~\ref{Pic_codim2_prop}~(i) below.
\end{proof}

\begin{rem}
Note that the blow-up of a normal scheme along a regular closed
embedding isn't necessarily normal if the subscheme is not reduced.
For example, take $X=\bbA^2={\rm Spec} (k[x,y])$ and $Z$ defined by
the equations $x^2=y^2=0$. Then, $\Bl$ is the subscheme of $\bbA^2
\times \bbP^1$ defined by the equations $x^2 v= y^2 u$ where $[u:v]$
are homogeneous coordinates for $\bbP^1$ and it is easy to check
that the whole exceptional fiber is
singular. Thus $\Bl$ is not normal (not even regular in codimension
one).
\end{rem}

\begin{prop}[Picard group of a projective bundle] \label{Pic_proj_bundle_prop}
Let $X\in Sch$ be a (connected) scheme and $\cF$ a vector bundle over $X$. We
consider the projective bundle $\bbP_X(\cF)$ associated to~$\cF$.
Its Picard group is $\Pic(X) \oplus \bbZ$ where $\bbZ$ is generated
by $\cO(-1)$ and $\Pic(X)$ comes from the pull-back from~$X$.
\end{prop}

\begin{proof}
Surjectivity of $\Pic(X)\oplus \bbZ\to \Pic(\bbP_X(\cF))$ is a
formal consequence of Quillen's formula~\cite[Prop.~4.3]{Quillen73}
for the K-theory of a projective bundle. Indeed, the determinant map
$K_0 \to \Pic$ is surjective with an obvious set theoretic section
and can easily be computed on each component of Quillen's formula.
Injectivity is obtained by pulling back to the fiber of a point for
the $\bbZ$ component, and by the projection formula for the
remaining $\Pic(X)$ component.
\end{proof}

\begin{prop}[Picard group of a blow-up] \label{Pic_codim2_prop}
Under Setup~\ref{setup}, we have\,:
\begin{enumerate}
\item[(i)]
The Picard group of $\Bl=\Bl_Z(X)$ is isomorphic to $\Pic(X)\oplus
\bbZ$ where the direct summand $\Pic(X)$ comes from the pull-back
$\pi^*$ and $\bbZ$ is generated by the class of the exceptional
divisor~$\Line(\Exc)=\cO_{\Bl}(-1)$.
\item[(ii)]
If $X$ is normal, the map $\upsilon^*:\Pic(X) \to \Pic(U)$ is injective. If $X$ is regular it is an isomorphism.
\item[(iii)]
The exceptional fiber $\Exc$ is the projective bundle $\bbP(C_{Z/X})$ over $Z$ and its Picard group is therefore $\Pic(Z)\oplus \bbZ$ where $\bbZ$ is generated by $\cO_{\Exc}(-1)$.
\item[(iv)]
The pull-back $\tilde{\iota}^*: \Pic(\Bl) \to \Pic(\Exc)$ maps
$[\Line(\Exc)] \in \Pic(\Bl)$ to $[\cO_{\Exc}(-1)]$.
\end{enumerate}
Under these identifications, Diagram~\eqref{blowup_eq} induces the following pull-back maps on
Picard groups\,:
$$
\xymatrix@C=4em{
\Pic(Z)
 \ar[d]_{\smallmatrice{\id \\ 0}}
& \Pic(X)
 \ar[l]_{\iota^*}
 \ar[r]^-{\upsilon^*}
 \ar[d]^{\smallmatrice{\id \\ 0}}
& \Pic(U)\,.\!
\\
\Pic(Z)\oplus \bbZ
& \Pic(X) \oplus \bbZ
 \ar[l]_{\smallmatrice{\iota^* & 0 \\ 0 & \id }}
 \ar[ru]_-{(\upsilon^* \; 0)}
}
$$
\end{prop}

\begin{proof}
By Example~\ref{L_Exc_exa}, we get~(iv) and we can deduce~(iii) from
Proposition~\ref{Pic_proj_bundle_prop}. To prove~(ii), use that for
$X$ normal (resp.\ regular) $\Pic(X)$ injects into (resp.\ is
isomorphic to) the group ${\rm Cl}(X)$ of Weil divisors classes
(see~\cite[7.1.19 and 7.2.14~(c), resp.\,7.2.16]{Liu02}), and that
${\rm Cl}(X)={\rm Cl}(U)$ since $\codim_X(Z)\geq 2$. Finally,
for~(i), consider the commutative diagram
\begin{equation}
\label{perf_eq}
\vcenter{\xymatrix{ \Dperf(Z) \ar[d]_-{\LL\tilde\pi^*}
& \Dperf(X) \ar[l]_-{\LL\iota^*} \ar[d]^-{\LL\pi^*}
\\
\Dperf(\Exc)
& \Dperf(\Bl) \ar[l]_-{\LL\tilde\iota^*}
}}
\end{equation}
of induced functors on the derived categories of perfect complexes.
We will use\,:

\noindent\textit{Fact 1\,:} The tensor triangulated functors
$\LL\pi^*$ and $\LL\tilde\pi^*$ are fully faithful with left inverse
$\RR\pi_*$ and $\RR\tilde\pi_*$ respectively, see
Thomason~\cite[Lemme~2.3]{Thomason93}.

\noindent\textit{Fact 2\,:} If $M\in\Dperf(\Bl)$ is such that $\LL\tilde\iota^*(M)\simeq\LL\tilde\pi^*(N)$ for some $N\in\Dperf(Z)$, then $M\simeq\LL\pi^*(L)$ for some $L\in\Dperf(X)$, which must then be $\RR\pi_*(M)$ by Fact~1. This follows from~\cite[Prop.\,1.5]{CHSW08}. (In their notation, our assumption implies that $M$ is zero in all successive quotients $\Dperf^{i+1}(\Bl)/\Dperf^{i}(\Bl)$ hence belongs to $\Dperf^0(\Bl)$.)

\smallbreak

Hence $\Pic(X)\oplus\bbZ\to\Pic(\Bl)$ is injective\,: If $L$ is a
line bundle on~$X$ and $n\in\bbZ$ are such that $\LL\pi^*(L)\otimes
\cO_{\Bl}(n)$ is trivial then we get $n=0$ by restricting to $\Exc$
and applying~(iii), and we get
$L\simeq\RR\pi_*\LL\pi^*L\simeq\RR\pi_*\cO_{\Bl}\simeq\RR\pi_*\LL\pi^*\cO_X\simeq\cO_X$
by Fact~1. So, let us check surjectivity of
$\Pic(X)\oplus\bbZ\to\Pic(\Bl)$. Let $M$ be a line bundle on~$\Bl$.
Using~(iii) again and twisting with $\cO_{\Bl}(n)$ if necessary, we
can assume that $\LL\tilde\iota^*(M)$ is isomorphic to
$\LL\tilde\pi^*N=\tilde\pi^*N$ for some line bundle $N$ on~$Z$. By
Fact~2, there exists $L\in\Dperf(X)$ such that $\LL\pi^*(L)\simeq
M$.
It now suffices to check that this $L\in\Dperf(X)$ is a line bundle.
The natural (evaluation) map $L^\vee\otimes L\to \cO_X$ is an
isomorphism, since it is so after applying the fully faithful tensor functor $\LL\pi^*:\Dperf(X)\to\Dperf(\Bl)$. So
$L\in\Dperf(X)$ is an invertible object, hence it is the
$m^\mathrm{th}$ suspension of a line bundle for $m\in\bbZ$,
see~\cite[Prop.~6.4]{Balmer07}. Using~\eqref{perf_eq}, one checks by
restricting to~$Z$ that $m=0$, \ie $L$ is a line bundle.
\end{proof}

\bigbreak \centerline{*\ *\ *}\medbreak

We now discuss dualizing complexes and relative canonical bundles. First of all, we mention the essential uniqueness of dualizing complexes on a scheme.

\begin{lem} \label{dualizing_lem}
Let $X\in \Sch$ be a scheme admitting a dualizing complex~$K_X$.
Then\,:
\begin{itemize}
\item[(i)] For any line bundle $L$ and any integer $i$, the complex $K_X \otimes L[i]$ is also a dualizing complex and any dualizing complex on $X$ is of this form.
\item[(ii)] If $K_X \otimes L[i] \simeq K_X$ in the derived category of~$X$, for some line bundle $L$ and some integer $i$, then $L \simeq \cO_X$ and $i=0$.
\end{itemize}
In other words, the set of isomorphism classes of dualizing complexes on $X$ is a principal homogeneous space under the action of $\Pic(X) \oplus \bbZ$.
\end{lem}
\begin{proof}
For (i), see~\cite[Lemma~3.9]{Neeman08_pre}. Let us prove (ii). We
have the isomorphisms
$$\cO_X \isoto {\RHom}(K_X,K_X) \simeq {\RHom}(K_X,K_X \otimes L[n]) \simeq {\RHom}(K_X,K_X)\otimes L[n] {\buildrel \sim\over\leftarrow} L[n]$$
in the coherent derived category. The first and last ones hold by
\cite[Prop.~3.6]{Neeman08_pre}. We thus obtain an isomorphism $\cO_X
\simeq L[n]$ in the derived category of perfect complexes (it is a
full subcategory of the coherent one). This forces $n=0$ and the
existence of an honest isomorphism of sheaves $\cO_X \simeq L$,
see~\cite[Prop.~6.4]{Balmer07} if necessary.
\end{proof}

We now use the notion of local complete intersection (l.c.i.)
morphism, that is, a morphism which is locally a regular embedding
followed by a smooth morphism, see~\cite[\S~6.3.2]{Liu02}. The
advantage of such morphisms $f:X'\to X$ is that $f^!$ is just $\LL
f^*$ twisted by a line bundle $\can_f$ and shifted by the relative
dimension~$\Kdim{f}$.

\begin{prop}
\label{f^!_prop}%
Let $f:X'\to X$ be an l.c.i.\ morphism with $X,X'\in\Sch$. Assume
that $f$ is proper. Then $f^!(\cO_X)$ is a shifted line bundle
$\can_f[\Kdim{f}]$ and there exists a natural isomorphism
$f^!(\cO_X) \otimes \LL f^*(-) \isoto f^!(-)$. In particular, $f^!$ preserves the subcategory $\Dperf$ of $\Dbc$.
\end{prop}

\begin{proof}
There is always a
natural morphism $f^!(\cO_X) \otimes \LL f^*(-) \to f^!(-)$. One shows
that it is an isomorphism and that $f^!(\cO_X)$ is a line bundle
directly from the definition, since both these facts can be checked
locally, are stable by composition and are true for (closed) regular
immersions and smooth morphisms by Hartshorne~\cite[Ch.~III]{Hartshorne66}. The subcategory $\Dperf$ is then preserved since both $\LL f^*$ and tensoring by a line bundle preserve it.
\end{proof}

The above proposition reduces the description of $f^!$ to that of
the line bundle~$\can_f$.

\begin{prop}
\label{concrete_can_prop}%
In the following cases, we have concrete descriptions of $\can_f$.
\begin{enumerate}
\item[(i)]
When $f:X'\to X$ is smooth and proper,
$\can_f\simeq\det(\Omega^1_{X'/X})$ is the determinant of the sheaf
of differentials. In particular, when $f$ is the projection of a
projective bundle $\bbP(\cF)$ to its base, where $\cF$ is a vector
bundle of rank~$r$, then $\can_f \simeq f^*({\rm det}\cF)\otimes
\cO_{\bbP(\cF)}(-r)$.
\smallbreak
\item[(ii)]
When $f:X'\hookrightarrow X$ is a regular closed immersion, $\can_f \simeq {\rm det} (N_{X'/X})$ is the determinant of the normal bundle. In particular when $f:E\hookrightarrow B$ is the inclusion of a prime divisor (Def.\,\ref{O(E)_def}), we have $\can_f \simeq \Line(E)\restr{E}$.
\end{enumerate}
\end{prop}

\begin{proof}
See~\cite[Prop.\,1 and Thm.\,3]{Verdier69}. See
alternatively~\cite[\S\,6.4.2]{Liu02}.
\end{proof}

\goodbreak

\begin{rem}
\label{lci_rem}%
All morphisms along which we consider push-forward in this article
are l.c.i.  It might not be obvious for $\pi: \Bl \to
X$ but this follows from~\cite[VII 1.8 p.~424]{SGA6} (it is locally of
the form mentioned there). So, $\can_\pi$ is also a line
bundle. Let us now describe the relative canonical line bundles in terms of $\can_\iota= {\rm det}(N_{Z/X})$.
\end{rem}

\begin{prop}
\label{can_prop}%
With the notation of Setup~\ref{setup}, we have
\begin{enumerate}
\smallbreak
\item[(i)]
\  $\can_{\tilde{\iota}}=\Line(\Exc)\restr{\Exc}=\cO_{\Exc}(-1)$
\smallbreak
\item[(ii)]
\  $\can_{\tilde{\pi}}=\tilde\pi^* \can_\iota^{\vee} \otimes
\Line(\Exc)\restr{\Exc}^{\otimes c}=\tilde\pi^* \can_\iota^{\vee}
\otimes \cO_{\Exc}(-c)$
\smallbreak
\item[(iii)]
\  $\can_{\pi}=\Line(\Exc)^{\otimes (c-1)}=\cO_{\Bl}(1-c).$
\end{enumerate}
By Proposition~\ref{f^!_prop}, it implies that we have
\begin{enumerate}
\item[(iv)] \  $\tilde\iota^!(-)=\Line(\Exc)\restr{\Exc} \otimes \LL\tilde\iota^*(-)[-1]=\cO_{\Exc}(-1) \otimes \LL\tilde\iota^*(-)[-1]$
\item[(v)] \  $\tilde\pi^!(-)=\tilde\pi^* \can_\iota^{\vee} \otimes \Line(\Exc)\restr{\Exc}^{\otimes c} \otimes \LL\tilde\pi^*(-)[c-1]=\tilde\pi^* \can_\iota^{\vee} \otimes \cO_{\Exc}(-c) \otimes \LL\tilde\pi^*(-)[c-1]$
\item[(vi)] \  $\pi^!(-)=\Line(\Exc)^{\otimes (c-1)} \otimes \LL \pi^*(-)=\cO_{\Bl}(1-c) \otimes \LL \pi^*(-).$
\end{enumerate}
\end{prop}
\begin{proof}
Points~(i) and~(ii) follow from Proposition~\ref{concrete_can_prop}~(ii) and~(i), respectively. They imply (iv) and (v). To prove point (iii) let us first observe that the exact sequence $\eqref{O(E)_eq}$ gives rise to an exact triangle
$$
\cO_{\Bl}(l+1) \to \cO_{\Bl}(l) \to \RR\tilde\iota_*(\cO_{\Exc}(l)) \to \cO_{\Bl}(l+1)[1]
$$
in $\Dperf(\Bl)$ for any $l \in \bbZ$. Applying $\RR \pi_*$ to this triangle and using that
$$
\RR \pi_* \RR \tilde\iota_*\cO_{\Exc}(l) = \RR \iota_* \RR \tilde\pi_*\cO_{\Exc}(l) = 0\quad \text{for} \quad -c< l<0
$$
(by~\cite[2.1.15]{EGA3-1}), we obtain by induction that $\RR \pi_*
\cO_{\Bl}(l)=\RR\pi_*\cO_{\Bl}= \cO_X$ for $-c<l\leq 0$. In
particular $\RR \pi_*\cO_{\Bl}(1-c)=\cO_X$. We now use the
filtration of~\cite[Prop.\,1.5]{CHSW08}. Let us show that
$\pi^!(\cO_X)\otimes \cO_{\Bl}(c-1)$ is in $\Dperf^0(\Bl)$. By
\loccit\ it suffices to show that $\LL
\tilde\iota^*(\pi^!(\cO_X)\otimes \cO_{\Bl}(c-1))$ is in
$\Dperf^0(\Exc)$. It follows from the sequence of isomorphisms
\begin{align*}
& \LL \tilde\iota^*(\pi^!(\cO_X)\otimes \cO_{\Bl}(c-1)) \simeq \LL \tilde\iota^* \pi^!(\cO_X) \otimes \cO_{\Exc}(c-1) \equalby{\simeq}{\text{(iv)}} \tilde\iota^! \pi^!(\cO_X) \otimes \cO_E(c)[1] \\
& \simeq \tilde\pi^! \iota^!(\cO_X) \otimes \cO_{\Exc}(c)[1] \equalby{\simeq}{\text{\ref{f^!_prop}}} \tilde\pi^!(\can_\iota) \otimes \cO_{\Exc}(c)[-c+1] \equalby{\simeq}{\text{(v)}} \cO_{\Exc}\,.
\end{align*}
Since $L:=\pi^!(\cO_X)\otimes \cO_{\Bl}(c-1)$ is in $\Dperf^0(\Bl)$, it is of the form $\LL \pi^* M$ for $M=\RR \pi_* L$ (by Fact~1 in the proof of Prop.~\ref{Pic_codim2_prop}) which we compute by duality:
\begin{align*}
& \RR \pi_* (\pi^!(\cO_X) \otimes \cO_{\Bl}(c-1)) \simeq \RR \pi_* \RHom(\cO_{\Bl},\pi^!(\cO_X)\otimes \cO_{\Bl}(c-1)) \simeq \\
& \simeq \RR \pi_* \RHom(\cO_{\Bl}(1-c), \pi^!(\cO_X)) \equalby{\simeq}{(\dagger)} \RHom(\RR \pi_* \cO_{\Bl}(1-c),\cO_X) \equalby{\simeq}{(\star)} \\
& \equalby{\simeq}{\mathstrut} \RHom(\cO_X,\cO_X) \simeq \cO_X
\end{align*}
where ($\star$) is by the computation at the beginning of the proof,
($\dagger$) is the duality isomorphism and all other isomorphisms
are obtained as consequences of the monoidal structure on the $\Dbc$
involved (see~\cite{Calmes09} and~\cite{Calmes08b_pre} for details).
Hence, $\pi^! (O_X) \simeq \cO_{\Bl}(1-c)$ as announced. This proves
(iii) and thus (vi).
\end{proof}

Finally, we also use dualizing complexes in the context of an
$\bbA^{\!*}$-bundle $U \to Y$, \ie a morphism that is locally of
the form $\bbA^n_Y \to Y$ (and is in particular flat).

\begin{lem} \label{affinebundle_lem}
Let $\alpha: U \to Y$ be an $\bbA^{\!*}$-bundle. Assume that $Y$
admits a dualizing complex~$K_Y$. Then
$\LL\alpha^*(K_Y)=\alpha^*(K_Y)$ is a dualizing complex on~$U$.
\end{lem}
\begin{proof}
This can be checked locally, so we can assume $\alpha$ decomposes as
$\alpha=f\circ u$ for an open immersion $u:\bbA^n_Y \hookrightarrow
\bbP^n_Y$ followed by the structural projection $f:\bbP^n_Y \to Y$.
Note that $\alpha$, $u$ and $f$ are all flat. We have by Propositions~\ref{f^!_prop} and~\ref{concrete_can_prop}\,(i)
$$
u^* f^! K_Y[n] = u^* (\cO(-n-1) \otimes f^* K_Y) \simeq u^* f^* K_Y = \alpha^* K_Y
$$
where the second equality comes from the triviality of $\cO(-n-1)$ on
$\bbA^n_Y$. Now $u^* f^! K_Y[n]$ is dualizing because proper
morphisms, open immersions and shifting preserve dualizing
complexes.
\end{proof}


\end{appendix}


\bigbreak\goodbreak
\noindent\textbf{Acknowledgments\,:} We thank Marc Levine for precious discussions on Theorem~\ref{standard_thm} and Example~\ref{Grass_exa}, Jean Fasel for a florilegium of relative bundles and Burt Totaro for several discussions on algebraic geometry.
\goodbreak


\providecommand{\bysame}{\leavevmode\hbox
to3em{\hrulefill}\thinspace}


\begin{thebibliography}{10}

\bibitem{SGA6}
\emph{Th\'eorie des intersections et th\'eor\`eme de
{R}iemann-{R}och},
  S\'eminaire de G\'eom\'etrie Alg\'ebrique du
  Bois-Marie 1966--1967 (SGA 6), dirig\'e par P.~Berthelot, A.~Grothendieck et
  L.~Illusie. Lecture Notes in Mathematics, Vol. 225. Springer-Verlag, Berlin, 1971.

\bibitem{Arason80}
J.~Arason, \emph{Der {W}ittring projektiver {R}\"aume},
Math. Ann.
  \textbf{253} (1980), no.~3, 205--212.

\bibitem{Balmer00}
P.~Balmer, \emph{{T}riangular {W}itt {G}roups {P}art 1: The
12-{T}erm
  {L}ocalization {E}xact {S}equence}, K-Theory \textbf{4} (2000), no.~19,
  311--363.

\bibitem{Balmer05b}
\bysame, \emph{Products of degenerate quadratic forms}, Compos.
Math.
  \textbf{141} (2005), no.~6, 1374--1404.

\bibitem{Balmer05a}
\bysame, \emph{Witt groups}, Handbook of $K$-theory. Vol. 2,
Springer, Berlin,
  2005, pp.~539--576.

\bibitem{BalmerCalmes08pp2}
P.~Balmer and B.~Calm\`es, \emph{Witt groups of {G}rassmann
varieties},
  preprint, 2008.

\bibitem{Balmer07}
P.~Balmer and G.~Favi, \emph{{G}luing techniques in triangular
geometry}, Q J
  Math \textbf{58} (2007), no.~4, 415--441.

\bibitem{Calmes08b_pre}
B.~Calm\`es and J.~Hornbostel, \emph{{P}ush-forwards for {W}itt
groups of
  schemes}, arXiv:0806.0571,
  2008, to appear in Comment. Math. Helv.

\bibitem{Calmes06_pre}
\bysame, \emph{Witt motives, transfers and d{\'e}vissage}, preprint,
available
  at \url{http://www.math.uiuc.edu/K-theory/0786/}, 2006.

\bibitem{Calmes09}
B.~Calmes and J.~Hornbostel, \emph{{T}ensor-triangulated categories
and
  dualities}, Theory Appl. Categ. \textbf{22} (2009), No. 6, 136--198.

\bibitem{CHSW08}
G.~Corti\~{n}as, C.~Haesemeyer, M.~Schlichting, and C.~Weibel,
\emph{Cyclic
  homology, cdh-cohomology and negative {K}-theory}, Ann. of Math. \textbf{167}
  (2008), 549--573.

\bibitem{Gille03b}
S.~Gille, \emph{Homotopy invariance of coherent {W}itt groups},
Math. Z.
  \textbf{244} (2003), no.~2, 211--233.

\bibitem{EGA3-1}
A.~Grothendieck, \emph{\'{E}l\'ements de g\'eom\'etrie alg\'ebrique.
{III}.
  \'{E}tude cohomologique des faisceaux coh\'erents. {I}}, Inst. Hautes
  \'Etudes Sci. Publ. Math. (1961), no.~11, 167.

\bibitem{Hartshorne66}
R.~Hartshorne, \emph{Residues and duality}, Lecture notes of a
seminar on the
  work of A. Grothendieck, given at Harvard 1963/64. With an appendix by P.
  Deligne. Lecture Notes in Mathematics, No. 20, Springer-Verlag, Berlin, 1966.

\bibitem{Hartshorne77}
\bysame, \emph{Algebraic geometry}, Springer-Verlag, New York, 1977,
Graduate
  Texts in Mathematics, No. 52.

\bibitem{Laksov72}
D.~Laksov, \emph{{A}lgebraic {C}ycles on {G}rassmann {V}arieties},
Adv. in
  Math. \textbf{9} (1972), no.~3, 267--295.

\bibitem{LevineMorel07}
M.~Levine and F.~Morel, \emph{Algebraic cobordism}, Springer
Monographs in
  Mathematics, Springer, Berlin, 2007.

\bibitem{Liu02}
Q.~Liu, \emph{Algebraic geometry and arithmetic curves}, Oxford
Graduate Texts
  in Mathematics, vol.~6, Oxford University Press, Oxford, 2002.

\bibitem{Neeman08_pre}
A.~Neeman, \emph{{D}erived categories and {G}rothendieck duality},
preprint 791
  of \url{http://www.crm.es/Publications/Preprints08.htm}.

\bibitem{Nenashev09}
A.~Nenashev, \emph{Projective push-forwards in the {W}itt theory of
algebraic
  varieties}, Adv. Math. \textbf{220} (2009), no.~6, 1923--1944.

\bibitem{Nenashev06}
A.~Nenashev and K.~Zainoulline, \emph{Oriented cohomology and
motivic
  decompositions of relative cellular spaces}, J. Pure Appl. Algebra
  \textbf{205} (2006), no.~2, 323--340.

\bibitem{Panin04}
I.~Panin, \emph{Riemann-{R}och theorems for oriented cohomology},
Axiomatic,
  enriched and motivic homotopy theory, NATO Sci. Ser. II Math. Phys. Chem.,
  vol. 131, Kluwer Acad. Publ., Dordrecht, 2004, pp.~261--333.

\bibitem{Quillen73}
D.~Quillen, \emph{Higher algebraic {$K$}-theory: {I}}, Algebraic
{$K$}-theory,
  Lecture Notes in Math., no. 341, Springer-Verlag, 1973, pp.~83--147.

\bibitem{Thomason93}
R.~W. Thomason, \emph{Les {$K$}-groupes d'un sch\'ema \'eclat\'e et
une formule
  d'intersection exc\'edentaire}, Invent. Math. \textbf{112} (1993), no.~1,
  195--215.

\bibitem{Verdier69}
J.-L. Verdier, \emph{Base change for twisted inverse image of
coherent
  sheaves}, Algebraic Geometry (Internat. Colloq., TIFR, Bombay, 1968), Oxford
  Univ. Press, London, 1969, pp.~393--408.

\bibitem{Walter03}
C.~Walter, \emph{{G}rothendieck-{W}itt groups of projective
bundles}, preprint,
  2003.

\end{thebibliography}
\end{document}